\documentclass[a4paper,11pt]{article}
\usepackage{amsmath}
\usepackage{amssymb}
\usepackage{theorem}
\usepackage{pstricks}
\usepackage{euscript}
\usepackage{epic,eepic}
\usepackage{graphicx}
\PassOptionsToPackage{normalem}{ulem}
\usepackage{amsmath}
\usepackage{amssymb}
\usepackage{theorem}
\usepackage{pstricks}
\usepackage{euscript}
\usepackage{epic,eepic}
\usepackage{graphicx}
\usepackage{setspace}
\PassOptionsToPackage{normalem}{ulem}

\newcommand{\menge}[2]{\big\{{#1} \;|\; {#2}\big\}} 

\newcommand{\Menge}[2]{\bigg\{{#1}~\bigg|~{#2}\bigg\}}

\newcommand{\emp}{\ensuremath{{\varnothing}}}

\newcommand{\scal}[2]{\left\langle{#1}\mid {#2} \right\rangle} 
\newcommand{\pscal}[2]{\langle\hspace{-0.2ex}\langle{#1}\mid{#2}\rangle\hspace{-0.2ex}\rangle} 

\newcommand{\vuo}{\ensuremath{\mbox{\footnotesize$\square$}}}

\newcommand{\nnm}{\vert\hspace{-0.2ex}\vert\hspace{-0.2ex}\vert}
\newcommand{\HH}{\ensuremath{\mathcal H}}
\newcommand{\GG}{\ensuremath{\mathcal G}}

\newcommand{\BL}{\ensuremath{\EuScript B}}

\newcommand{\HHH}{\ensuremath{\boldsymbol{\mathcal H}}}
\newcommand{\WW}{\ensuremath{\boldsymbol{W}}}
\newcommand{\KKK}{\ensuremath{\boldsymbol{\mathcal K}}}
\newcommand{\GGG}{\ensuremath{\boldsymbol{\mathcal G}}}
\newcommand{\VV}{\ensuremath{\boldsymbol{V}}}
\newcommand{\CCC}{\ensuremath{\boldsymbol{C}}}
\newcommand{\MM}{\ensuremath{\boldsymbol{M}}}

\newcommand{\RR}{\ensuremath{\mathbb R}}
\newcommand{\KK}{\ensuremath{\mathcal K}}
\newcommand{\RP}{\ensuremath{\left[0,+\infty\right[}}
\newcommand{\RPP}{\ensuremath{\,\left]0,+\infty\right[}}

\newcommand{\NN}{\ensuremath{\mathbb N}}

\newcommand{\dom}{\ensuremath{\operatorname{dom}}}

\newcommand{\prox}{\ensuremath{\operatorname{prox}}}

\newcommand{\cart}{\ensuremath{\mbox{\huge{$\times$}}}}

\newcommand{\TT}{\ensuremath{{\mathbf T}}}

\newcommand{\rh}{\ensuremath{{\mathrm  a}}}
\newcommand{\og}{\ensuremath{{\mathrm  b}}}
\newcommand{\rk}{\ensuremath{{\mathsf {\mathbf  a}}}}
\newcommand{\ck}{\ensuremath{{\mathsf {\mathbf  u}}}}
\newcommand{\xk}{\ensuremath{{\mathsf{\mathbf  x}}}}
\newcommand{\yk}{\ensuremath{{\mathsf{\mathbf  y}}}}

\newcommand{\argmin}{\ensuremath{\operatorname{argmin}}}
\newcommand{\ran}{\ensuremath{\operatorname{ran}}}
\newcommand{\zer}{\ensuremath{\operatorname{zer}}}
\newcommand{\gra}{\ensuremath{\operatorname{gra}}}

\newcommand{\vv}{\ensuremath{\boldsymbol{v}}}
\newcommand{\sss}{\ensuremath{\boldsymbol{s}}}
\newcommand{\xx}{\ensuremath{\boldsymbol{x}}}
\newcommand{\pp}{\ensuremath{\boldsymbol{p}}}
\newcommand{\qq}{\ensuremath{\boldsymbol{q}}}
\newcommand{\yy}{\ensuremath{\boldsymbol{y}}}

\newcommand{\rr}{\ensuremath{\boldsymbol{r}}}

\newcommand{\zz}{\ensuremath{\boldsymbol{z}}}
\newcommand{\bb}{\ensuremath{\boldsymbol{b}}}

\newcommand{\cc}{\ensuremath{\boldsymbol{c}}}
\newcommand{\dd}{\ensuremath{\boldsymbol{d}}}
\newcommand{\aaa}{\ensuremath{\boldsymbol{a}}}
\newcommand{\ww}{\ensuremath{\boldsymbol{w}}}
\newcommand{\BB}{\ensuremath{\boldsymbol{B}}}
\newcommand{\LL}{\ensuremath{\boldsymbol{L}}}
\newcommand{\PPP}{\ensuremath{\mathsf{P}}}
\newcommand{\UU}{\ensuremath{\mathbf{U}}}

\newcommand{\E}{\ensuremath{\mathsf{E}}}

\newcommand{\AAA}{\ensuremath{\boldsymbol{A}}}
\newcommand{\BBB}{\ensuremath{\boldsymbol{B}}}
\newcommand{\QQ}{\ensuremath{\boldsymbol{Q}}}
\newcommand{\SSS}{\ensuremath{\boldsymbol{S}}}
\newcommand{\DD}{\ensuremath{\boldsymbol{D}}}

\newcommand{\FF}{\ensuremath{\EuScript F}}

\newcommand{\Id}{\ensuremath{\operatorname{Id}}}

\newcommand{\weakly}{\ensuremath{\rightharpoonup}}

\newtheorem{theorem}{Theorem}[section]
\newtheorem{lemma}[theorem]{Lemma}

\newtheorem{corollary}[theorem]{Corollary}

\theoremstyle{plain}{\theorembodyfont{\rmfamily}
}
\theoremstyle{plain}{\theorembodyfont{\rmfamily}
}
\theoremstyle{plain}{\theorembodyfont{\rmfamily}
\newtheorem{algorithm}[theorem]{Algorithm}}
\theoremstyle{plain}{\theorembodyfont{\rmfamily}
\newtheorem{example}[theorem]{Example}}
\theoremstyle{plain}{\theorembodyfont{\rmfamily}
\newtheorem{problem}[theorem]{Problem}}
\theoremstyle{plain}{\theorembodyfont{\rmfamily}
\newtheorem{remark}[theorem]{Remark}}
\theoremstyle{plain}{\theorembodyfont{\rmfamily}
}

\definecolor{labelkey}{rgb}{0,0.08,0.45}
\definecolor{refkey}{rgb}{0,0.6,0.0}
\definecolor{Brown}{rgb}{0.45,0.0,0.05}
\definecolor{dgreen}{rgb}{0.00,0.49,0.00}
\definecolor{dblue}{rgb}{0,0.08,0.75}
\numberwithin{equation}{section}

\begin{document}
\title{\sffamily 
A stochastic  inertial forward-backward splitting algorithm for multivariate monotone inclusions
\thanks{This material is based upon work supported by the Center for Brains, Minds and Machines (CBMM), funded by NSF STC award CCF-1231216.
L. Rosasco acknowledges the financial support of the Italian Ministry of Education, University and Research FIRB project RBFR12M3AC.
S. Villa is member of the Gruppo Nazionale per
l'Analisi Matematica, la Probabilit\`a e le loro Applicazioni (GNAMPA)
of the Istituto Nazionale di Alta Matematica (INdAM). 
Bang Cong Vu's research work is partially funded by Vietnam National Foundation for Science
and Technology Development (NAFOSTED) under Grant No. 102.01-2014.02.}}
\author{Lorenzo Rosasco$^{1,2}$, Silvia Villa$^2$ and
B$\grave{\text{\u{a}}}$ng C\^ong V\~u$^2$
\\[5mm]
\small
\small $\!^1$ DIBRIS, Universit\`a degli Studi di Genova\\
\small  Via Dodecaneso 35, 16146, Genova, Italy\\
\small \ttfamily{lrosasco@mit.edu}
\\[5mm]
\small
\small $\!^2$ LCSL, Istituto Italiano di Tecnologia\\
\small and Massachusetts Institute of Technology,\\
\small Bldg. 46-5155, 77 Massachusetts Avenue, Cambridge, MA 02139, USA\\
\small \ttfamily{\{Silvia.Villa,Cong.Bang\}@iit.it}
}
\date{~}

\maketitle
\begin{abstract}
We propose an inertial forward-backward splitting algorithm to compute the 
zero of a sum of two monotone operators allowing for stochastic errors in the computation of the operators. 
More precisely, we establish almost sure convergence in real Hilbert spaces of the sequence of 
iterates to an optimal solution. Then, based on this analysis, we introduce two new classes of  
stochastic inertial primal-dual splitting methods for solving structured 
systems of composite monotone inclusions and prove their convergence. 
Our results extend  to the stochastic and inertial setting various types of structured monotone inclusion problems and corresponding algorithmic solutions. 
Application to minimization problems is discussed.
\end{abstract}





\section{Introduction}
\label{intro}
A wide class of problems reduces to the problem of finding a zero point of the sum of a maximally monotone operator $A$  and a cocoercive operator $B$ acting on a real Hilbert space $\HH$.
Problems of the above form arise  in diverse areas of applied mathematics, including
partial differential equations \cite{Sib70},  mechanics and evolution inclusions \cite{plc2010,Glow89,Har81,mercier79,Merc80}, signal and image processing and inverse problems \cite{siam05,Dau04},  convex optimization, statistics and learning theory \cite{Devi11,Duch09,MRSVV10,RosVilMos13,VilSal13}, game theory \cite{Bric13}, variational inequalities \cite{Facc03,Tseng90,Tseng91,Zhud96}, and stochastic optimization \cite{Nem09,Barty07,Bennar07}.
 One of the most popular approaches to approximate a solution  
is the forward-backward splitting method \cite{plc04,siam05,mercier79}. \\
The extension to the case of variable metric and to preconditioning has been considered in \cite{CheRoc97,optim2}.
This extension is crucial, since preconditioned forward-backward splitting 
can be used to solve a broad class of structured composite monotone inclusion problems
in duality, by formulating them as instances of the above fundamental monotone inclusion in product Hilbert spaces.
Indeed,  within this framework it is possible  to recover several primal-dual splitting methods proposed in the literature, see \cite{optim2,Con13,icip14} for details. 
This basic procedure has been extended by using the product space reformulation technique to solve coupled  systems of monotone inclusions in \cite{plc2010} and then in \cite{jota1}. \\
Inspired by the accelerated gradient method of Nesterov \cite{Nes83}, inertial variants of  forward-backward splitting for solving monotone inclusions 
have been introduced in \cite{Dirk} (see also \cite{Polyak64,Attouch01,Moudafi03,JC10}).
In particular, \cite{Dirk} discusses the derivation of inertial primal-dual algorithms from the inertial forward-backward algorithm applied to suitable monotone inclusions in duality. 

The goal of the paper is to extend this analysis to the stochastic setting. Recently, stochastic versions of splitting methods for monotone inclusions, such as  stochastic
forward-backward splitting \cite{plc14,LSB14b},  stochastic Douglas-Rachford \cite{plc14}, and 
stochastic versions of primal-dual methods as in \cite{BiaHacIut14,plc14,pesquet14} have been proposed. These works have found 
applications to stochastic optimization \cite{plc14,LSB14b} and machine learning  \cite{Duch09,LSB14a}. In this paper, we  propose and study a stochastic inertial forward-backward splitting algorithm for  solving the following  monotone inclusion.

\begin{problem}
\label{prob1}
Let $\beta\in\RPP$,
let $\HH$ be a real Hilbert space. 
Let $U \in \BL(\HH)$ be self-adjoint and  strongly positive,
 let $A\colon \HH\to 2^{\HH}$ be maximally monotone, 
let $B\colon \HH\to\HH$ be such that
 for every  $(x,y)\in\HH^2$,
\begin{alignat}{2}
\label{coco} 
 \scal{x-y}{Bx - By} \geq \beta 
\scal{Bx - By }{U(Bx - By)}.
\end{alignat}
Suppose that the set $\mathcal{P}$ of  all points
$\overline{x}\in\HH$  such that
 \begin{equation}
\label{e:pri}
0\in A\overline{x} 
+ B\overline{x} 
\end{equation}
is non-empty. 
The problem is to find a point in $\mathcal{P}$.
\end{problem}

We  show that the above  inclusion includes as special cases coupled systems of
monotone inclusions, arising in the study of evolution inclusions, variational problems, 
best approximation, and network flows. We refer the reader to \cite{plc2010} for a discussion of several applications. 
Our main result establishes almost sure convergence of the iterates of the considered algorithm. 
Such a result builds on ideas introduced  in \cite{optim2} and \cite{icip14}. As 
a corollary it allows to derive, as special  cases, two new classes of stochastic inertial primal-dual splitting methods for solving coupled system  of composite monotone inclusions involving parallel sums.

The rest of the paper is organized as follows.
We recall some notation and background 
on monotone operator theory in Section \ref{backgr}. 
Then, in Section \ref{Algocon}, we define the stochastic inertial forward-backward splitting algorithm solving 
Problem \ref{prob1} and analyze its convergence. 
In Section \ref{s:cop}, the application to coupled systems of monotone inclusions in duality, and 
minimization problems is derived, Finally, the derivation of two classes of  stochastic inertial primal-dual splitting methods
is proposed in Section \ref{Algocon}.

\section{Notation--background and preliminary results}
\label{backgr}
Throughout, $\HH$  is a real separable Hilbert space.
We denote by  $\scal{\cdot}{\cdot}$ and $\|\cdot\|$  the scalar product and the associated norm
of $\HH$.  
The symbols $\weakly$ and $\to$ denote weak and strong convergence, respectively. 
We denote  by $\ell_+^1(\NN)$ the set of summable sequences in  $\RP$, and 
by $\BL(\HH)$ the space of linear operators from $\HH$ into itself. 
Let $U\in\BL(\HH)$ be self-adjoint and strongly positive, i.e. 
\begin{equation}
(\exists \chi\in\RPP)(\forall x\in\HH) \quad \scal{Ux}{x}\geq \chi\|x\|^2.
\end{equation} 
We define a scalar product and a norm respectively by
\[(\forall x\in\HH)(\forall y\in\HH)\quad\scal{x}{y}_U=\scal{Ux}{y}
\quad\text{and}\quad\|x\|_U=\sqrt{\scal{Ux}{x}}.
\]
Let $A\colon\HH \to 2^\HH$ be a set-valued operator.
The domain and the graph of $A$ are defined by  
$$\dom A =\menge{x\in\HH}{Ax\neq\emp}\;\text{and}\;
\gra A =\menge{(x,u)\in\HH\times\HH}{u\in Ax}.$$
The set of zeros of $A$ is $\zer A=\menge{x\in\HH}{0\in Ax}$
and the range of $A$ is  $\ran A=A(\HH)$.
The inverse of $A$  is $A^{-1}\colon\HH \to 2^\HH\colon u\mapsto 
\menge{x\in\HH}{u\in Ax}$. The resolvent of $A$ is 
\begin{equation}
\label{eq:res}
J_A=(\Id+A)^{-1},
\end{equation}
where $\Id$ denotes the identity operator of $\HH$. Moreover, 
$A$ is monotone if
\[
(\forall (x,u)\in\gra A) (\forall (y,v)\in\gra A)\quad \scal{x-y}{u-v} \geq 0, 
\]
and maximally so, if there exists no monotone operator $\widetilde{A}\colon\HH\to\HH$
such that $\gra A\subset \gra \widetilde{A}\neq \gra A$.
Let $T\colon\HH\to\HH$. Then $T$ is firmly nonexpansive if 
\begin{equation}
(\forall (x,y)\in\HH^2)\quad \|Tx-Ty\|^2\leq \|x-y\|^2-\|(\Id-T)x-(\Id-T)y\|^2.
\end{equation}
If  $A$ is monotone, then $J_A$ is single-valued and firmly nonexpansive, and, 
in addition, if $A$ is maximally monotone, then $\dom J_A=\HH$.
The parallel sum of $A\colon\HH\to 2^\HH$ and 
$B\colon\HH \to 2^\HH$ is
\[
 A\;\vuo\; B = (A^{-1}+ B^{-1})^{-1}.
\]
$A$ is demiregular at $y\in\dom A$ if, for every sequence $(y_n,v_n)_{n\in\NN}$ in $\gra A$ and every 
$v\in Ay$, we have $(y_n\rightharpoonup y,\, v_n\to v) \implies y_n\to y$.

 Let $\Gamma_0(\HH)$ be the class of proper lower semicontinuous convex functions from $\HH$ to
$\left]-\infty,+\infty\right]$. For any self-adjoint strongly positive operator
$U\in \BL(\HH)$ and $f\in\Gamma_0(\HH)$, we define
 \begin{equation}
\label{set0}
\prox_{f}^{U}\colon\HH\to\HH \colon x\mapsto\underset{y\in\HH}{\argmin}\: 
\big( f(y) + \frac{1}{2}\|x-y\|_{U}^2\big),
\end{equation}
and 
 \[
\prox_{f}\colon\HH\to\HH \colon x\mapsto\underset{y\in\HH}{\argmin}\: 
\big( f(y) + \frac{1}{2}\|x-y\|^2\big).
\]
It holds $\prox_{f}^U=J_{U^{-1}\partial f}$, and $\prox_{f}= J_{\partial f}$
coincides with the classical definition of proximity operator in \cite{Mor62}.
The  conjugate function of $f$ is 
\[
 f^*\colon a\mapsto \sup_{x\in\HH}\big(\scal{a}{x}-f(x)\big).
\]
Note that, 
\[
 (\forall f\in\Gamma_0(\HH))(\forall x\in\HH)(\forall y\in\HH)\quad y\in\partial f(x) \Leftrightarrow x\in\partial f^*(y),
\]
or equivalently,
\begin{equation}
\label{e:equi}
 (\forall f\in\Gamma_0(\HH))\quad (\partial f)^{-1} = \partial f^*.
\end{equation}
The infimal convolution of the two functions $f$ and $g$ from $\HH$ to $\left]-\infty,+\infty\right]$ is 
\[
 f\;\vuo\; g\colon x \mapsto \inf_{y\in\HH}(f(y)+g(x-y)).
\]
 The strong relative interior of a subset $C$ of $\HH$ is 
the set of points $x\in C$ such that the cone generated by 
$-x + C$ is a closed vector subspace of $\HH$. We refer to \cite{livre1} for an account of the main 
results of convex analysis, monotone operator theory, and the theory of nonexpansive operators in the context 
of Hilbert spaces.

Let $(\Omega, \FF,\mathsf{P})$ be a probability space.
A $\HH$-valued random variable is a measurable function $X\colon \Omega\to\HH$, 
where $\HH$ is endowed with the Borel $\sigma$-algebra. We denote by 
$\sigma(X)$ the  $\sigma$-field generated by $X$.
The expectation of a random variable $X$ is denoted by $\E[X]$. The conditional expectation
of $X$ given a $\sigma$-field $\EuScript{A}\subset \EuScript{F}$ is denoted by 
$\E[X|\EuScript{A}]$. Given a random variable $Y\colon\Omega\to\HH$,
the conditional expectation of $X$ given $Y$, that is $\E[X|\sigma(Y)]$ is denoted by $\E[X|Y]$. See \cite{LedTal91} for more
details on probability theory in Hilbert spaces.  A $\HH$-valued random process is a sequence $(x_n)_{n\in\NN}$ 
of $\HH$-valued random variables.
The abbreviation a.s. stands for ``almost surely''.
\begin{lemma}{\rm\cite[Theorem 1]{Rob85}}\label{l:rob85}
Let $(\FF_n)_{n\in\NN}$ be an increasing sequence of  sub-$\sigma$-algebras of ${\EuScript{F}}$,
let $(z_n)_{n\in\NN}$, $(\xi_n)_{n\in\NN}$, $(\zeta_n)_{n\in\NN}$ and $(t_n)_{n\in\NN}$ be 
$\RP$-valued random sequences such that,  for every $n\in\NN$, $z_n$, $\xi_n$, $\zeta_n$,
and $t_n$  are $\FF_n$-measurable. Assume moreover that  $\sum_{n\in\NN}t_n<+\infty$,
$\sum_{n\in\NN} \zeta_n<+\infty$ a.s.,  and 
\begin{equation}\label{eq:rob}
(\forall n\in\NN)\quad \E[z_{n+1}|\FF_n] \leq (1+t_n)z_n + \zeta_n-\xi_n 
\quad \text{a.s.}.
\end{equation}
Then $(z_n)_{n\in\NN}$ converges a.s. and $(\xi_n)_{n\in\NN}$ is summable a.s..
\end{lemma}
The following lemma is a special case of \cite[Proposition 2.3]{plc14}.
\begin{lemma}
\label{p:fejer}
 Let $C$ be a non-empty closed subset of $\HH$ and
let  $(x_n)_{n\in\NN}$ be a  $\HH$-valued random process.
For every $n\in\mathbb{N}$, set $\EuScript{F}_n=\sigma(x_0,\ldots,x_n)$. 
Suppose that, for every $x\in C$, there exist $\RP$-valued random 
 sequences $(\xi_n(x))_{n\in\NN}$, $(\zeta_n(x))_{n\in\NN}$ and $(t_n(x))_{n\in\NN}$
such that,  for every $n\in\NN$, $\xi_n(x)$, $\zeta_n(x)$ and $t_n(x)$
are ${\EuScript{F}}_n$-measurable, $(\zeta_n(x))_{n\in\NN}$ and $(t_n(x))_{n\in\NN}$ are summable a.s.,
and
\begin{equation}\label{eq:rob}
(\forall n\in\NN)\quad \E[\|x_{n+1}-x\|^2|\FF_n] \leq (1+t_n(x))\|x_{n}-x\|^2 + \zeta_n(x)-\xi_n(x)
\quad \text{a.s.}
\end{equation}
Then the following hold.
\begin{enumerate}
\item 
\label{p:fejerii} $(x_n)_{n\in\NN}$ is bounded a.s.
\item 
\label{p:fejeri}  There exists $\widetilde{\Omega}\subset \Omega$ such that $\PPP(\widetilde{\Omega})=1$ 
and, for every $\omega\in\widetilde{\Omega}$ and $x\in C$,
 $( \|x_{n}(\omega)-x\|)_{n\in\NN}$ converges a.s.
\item 
\label{p:fejeriv}
Suppose that the set of weak cluster points of $(x_n)_{n\in\NN}$ is a subset of $C$ a.s. Then
$(x_n)_{n\in\NN}$ converges weakly a.s. to a $C$-valued random vector.
\end{enumerate}
\end{lemma}
\begin{lemma}{\rm \cite[Lemma 3.7]{optim2}}
\label{l:maxmon45}
Let $A\colon\HH \to 2^{\HH}$ be maximally monotone, 
 let $U\in\BL(\HH)$ be self-adjoint and strongly positive, and let 
$\GG$ be the real Hilbert space obtained by endowing $\HH$ with 
the scalar product 
$(x,y)\mapsto\scal{x}{y}_{U^{-1}}=\scal{x}{U^{-1}y}$.
Then, the following hold.
\begin{enumerate}
\item
\label{l:maxmon45i}
$UA\colon\GG \to 2^{\GG}$ is maximally monotone.
\item
\label{l:maxmon45ii}
$J_{UA}\colon\GG\to\GG$ is firmly nonexpansive.
\end{enumerate}
\end{lemma}
\section{Main results}
\label{Algocon}
In this section we introduce the stochastic 
inertial forward-backward algorithm for solving Problem \ref{prob1} and 
analyze its convergence behavior. We recall that  $\beta$ is the constant 
defined in \eqref{coco}.
\begin{algorithm}
\label{a:al1}
Let $\varepsilon\in\left]0,\min\{1,\beta\}\right[$,
let $(\gamma_n)_{n\in\NN}$ be a sequence in 
$\left[\varepsilon,(2-\varepsilon)\beta\right]$, let $(\lambda_n)_{n\in\NN}$ be a sequence in  $\left[\varepsilon,1\right]$,
and  let $(\alpha_n)_{n\in\NN}$ be a sequence in  $\left[0,1-\varepsilon\right]$.
Let $(\rr_{n})_{n\in\NN}$  be  a $\HH$-valued, square integrable random process,
let $x_{0}$ be  a $\HH$-valued, squared integrable random variable and set $x_{-1}=x_{0}$. Furthermore,
set 
\begin{equation}
\label{e:main1*}
(\forall n\in \NN)\quad
\begin{array}{l}
\left\lfloor
\begin{array}{l} 
w_{n}=x_{n}+\alpha_n  (x_n-x_{n-1}) \\
z_n = w_n-\gamma_n U \rr_n\\
 p_{n}= J_{\gamma_n UA}(z_{n})\\
x_{n+1}=x_{n}+\lambda_{n} (p_{n}-x_{n}).\\
\end{array} 
\right.\\[2mm]
\end{array}
\end{equation}

\end{algorithm}
\begin{theorem}
\label{t:1} 
Consider Algorithm~\ref{a:al1}, and set $(\forall n\in\NN)\; \FF_n = \sigma(x_{0}, \ldots, x_{n})$.  
Suppose that  the following conditions are satisfied.
\begin{enumerate}
\item\label{cond:one1} 
$ (\forall n\in \NN)\;
\E[\rr_{n}| \FF_n] = Bw_{n}
$ a.s.
\item 
\label{cond:two2} 
$\sum_{n\in\NN}
\E[ \|  r_{n}-Bw_{n}\|^2 | \FF_n] < +\infty$ a.s.
\item\label{cond:three3}
$\sup_{n\in\NN} \| x_{n}-x_{n-1}\|^2 <\infty $ a.s.
and $\sum_{n\in\NN}\alpha_n < +\infty $ a.s.
\end{enumerate}
Then, the following hold for some  a.s. $\mathcal{P}$-valued random variable 
$\overline{x}$.
\begin{enumerate}
\item \label{t:1i}  $x_{n}\weakly \overline{x}$ a.s.
\item\label{t:1ia}
$Bx_{n}\to B\overline{x}$ a.s.
\item\label{t:1ii}
 If $B$ is demiregular at $\overline{x}$, then
$ x_{n}\to \overline{x}$ a.s.
\end{enumerate}
\end{theorem}
\begin{proof}
Let $x\in\mathcal{P}$ and set
\begin{equation}
\label{e:set1}(\forall n\in \NN)\quad
u_n = w_n-p_n - \gamma_nU(\rr_n- B x).
\end{equation}
Since
\begin{equation}
(\forall n\in \NN)\quad
x_{n+1} = (1-\lambda_n)  x_n + \lambda_n p_n,
\end{equation}
then, upon setting  $V = U^{-1}$, and using the convexity of  $\|\cdot\|_{V}^{2}$,  we obtain
\begin{equation}
\label{e:sett}(\forall n\in \NN)\quad
\|x_{n+1}-x\|_{V}^2 \leq (1-\lambda_n)  \|x_{n}-x\|_{V}^2
+\lambda_n \|p_{n}-x\|_{V}^2.
\end{equation}
Since $x\in\mathcal{P}$, we have 
\begin{equation}
(\forall n\in \NN)\quad
x = J_{\gamma_n U A}(x - \gamma_nU B x).
\end{equation}
By Lemma \ref{l:maxmon45}(ii), $J_{\gamma_n UA}$ is firmly nonexpansive
with respective to  $\|\cdot \|_{V}$, 
and therefore
\begin{alignat}{2}
\label{e:est1}(\forall n\in \NN)\quad
\|p_n-x \|_{V}^2 &\leq \|w_n- x -\gamma_nU(\rr_n- B x) \|_{V}^2
-  \|u_n \|_{V}^2\notag\\
& = \|w_n- x\|_{V}^2 - 2\gamma_{n}\scal{w_n- x}{\rr_n- B x}\notag\\
&\hspace{3cm}
+ \gamma_{n}^2 \|U(\rr_n- B x) \|_{V}^2 -  \|u_n \|_{V}^2.
\end{alignat}
Using \ref{cond:one1}, since $w_n$ is $\FF_n$-measurable, we have 
\begin{alignat}{2}
\label{e:est2}(\forall n\in \NN)\quad
\E[\scal{w_n- x}{\rr_n- B x}|\FF_n] 
&= 
\scal{w_n- x}{\E[\rr_n|\FF_n]- B x}\notag\\
&= \scal{w_n- x}{Bw_n- B x}.
\end{alignat}
By the same reason, for every $n\in\NN$, $Bw_n$ is $\FF_n$-measurable, and we also have 
\begin{alignat}{2}
\label{e:est3}
\E[\|U(\rr_n- B x) \|_{V}^2|\FF_n] 
&= \E[\|U(\rr_n- B w_n) \|_{V}^2|\FF_n] +\|U(B w_n- B x) \|_{V}^2\notag\\
&\hspace{1.4cm}+ 2\E[\scal{Bw_n- Bx}{\rr_n- Bw_n}|\FF_n ] \notag\\
&= \E[\|U(\rr_n- B w_n) \|_{V}^2|\FF_n] +\|U(B w_n- B x) \|_{V}^2\notag\\
&\hspace{1.4cm}+ 2\scal{Bw_n- Bx}{\E[\rr_n|\FF_n]- Bw_n}\notag\\
&=\E[\|U(\rr_n- B w_n) \|_{V}^2|\FF_n] +\|U(B w_n- B x) \|_{V}^2\notag\\
&\leq \E[\|U(\rr_n- B w_n) \|_{V}^2|\FF_n]  +
\beta^{-1}\scal{w_n- x}{B w_n- B x},
\end{alignat}
where the last inequality follows from \eqref{coco1}.
Therefore, for every $n\in\NN$, we derive from~\eqref{e:est1},~\eqref{e:est2}, and~\eqref{e:est3} that 
\begin{alignat}{2}
\E[\|\pp_n-  x \|_{V}^2|\FF_n]
& \leq  \|w_n- x\|_{V}^2 
- \gamma_n(2-\beta^{-1}\gamma_n)\scal{w_n- x}{B w_n- B x}\notag\\
&\quad+\gamma_{n}^2\E[\|U(\rr_n- B w_n) \|_{V}^2|\FF_n] -\E[\|U_n \|_{V}^2|\FF_n] \notag\\
&\leq \|w_n- x\|_{V}^2 -\varepsilon\gamma_n\scal{w_n- x}{B w_n- B x}\notag\\
&\hspace{0.5cm}+\gamma_{n}^2\E[\|U(\rr_n- B w_n) \|_{\VV}^2|\FF_n] -\E[\|U_n \|_{V}^2|\FF_n]\notag\\
&\leq \|x_n- x\|_{V}^2 +\alpha_n( \|x_n- x\|_{V}^2 - \|x_{n-1}- x\|_{V}^2)
+\zeta_n-\xi_n,
\label{e:cc1}
\end{alignat}
where
\begin{equation}
\label{eq:dddd}
(\forall n\in\NN)\quad
\begin{cases}
\zeta_n = 2\alpha_n \|x_{n-1}-x_n\|_{V}^2 + \gamma_{n}^2\E[\|U(\rr_n- B w_n) \|_{V}^2|\FF_n] \\
\xi_n= \E[\|U_n \|_{V}^2|\FF_n]+ \varepsilon\gamma_n\scal{w_n- x}{Bw_n- B x}.
\end{cases}
\end{equation}
Using  \eqref{e:sett} and \eqref{e:cc1},
we obtain,
\begin{alignat}{2}
\label{e:ccc2}
(\forall n\in\NN)\quad\E[\|x_{n+1}- x\|_{V}^2 |\FF_n] 
&\leq (1-\lambda_n)\|x_n- x\|_{V}^2 + \lambda_n \E[\|p_n- x\|_{V}^2|\FF_n]\notag \\
&\leq (1+\alpha_n)\|x_n- x\|_{V}^2 + \zeta_n - (\alpha_n \|x_{n-1}- x\|_{V}^2 +\xi_n).
\end{alignat} 
By~\eqref{eq:dddd} and since $B$ is monotone, for each $n\in\NN$,  $\zeta_n$ and $\xi_n$ are non-negative and 
$\FF_n$-measurable. By (b1) and (c1), $(\zeta_n)_{n\in\NN}$ is summable, and
hence, we derive from Lemma \ref{l:rob85} that 
\begin{equation}\label{eq:aa}
\exists\; \tau =  \lim_{n\to\infty} \|x_{n}- x\|_{V}^2 \quad 
\text{and}
\quad \sum_{n\in\NN}(\alpha_n \|x_{n-1}- x\|_{V}^2 +\xi_n) < +\infty.
\end{equation}
Moreover,
since $\inf_{n\in\NN}\gamma_n\geq\epsilon>0$, we also have
\begin{equation}\label{eq:co1}
\sum_{n\in\NN}\scal{w_n- x}{Bw_n- B x} < +\infty
\Longrightarrow \scal{w_n- x}{Bw_n- B x} \to 0.
\end{equation} 
and 
\begin{equation}\label{eq:co2}
\sum_{n\in\NN} \E[\|U_n \|^2|\FF_n]< +\infty 
\Longrightarrow\E[\| w_n-p_n - \gamma_nU(\rr_n- B x)\|^2|\FF_n] \to 0.
\end{equation} 
Next, using~\eqref{coco1}, we derive from~\eqref{eq:co1} that 
\begin{equation}\label{eq:co3}
 Bw_n \to B x.
\end{equation} 
We also derive from~\eqref{eq:co2},  (b1), and~\eqref{eq:co3}  that
\begin{alignat}{2}\label{eq:co4}
\E[\| w_n-p_n\|^2|\FF_n] &\leq 2 \E[\| w_n-p_n - \gamma_nU(\rr_n- B x)\|^2|\FF_n] 
+ 2\E[\| \gamma_nU(\rr_n- B x)\|^2|\FF_n] \notag\\
&\leq2\bigg(\E[\| w_n-p_n - \gamma_nU(\rr_n- B x)\|^2|\FF_n] \bigg) + 
4\E[\| \gamma_nU(\rr_n- B w_n)\|^2|\FF_n] \notag\\
&\quad+4 \| \gamma_nU(Bw_n- B x)\|^2\to 0.
\end{alignat} 
Hence, since $\inf_{n\in\NN}\gamma_n>0$, we obtain 
\begin{equation}\label{eq:co6}
 \E[\|\rr_n-Bx \|^2|\FF_n] \to 0.
\end{equation}
Now define 
\begin{equation}
(\forall n\in\NN)\quad\overline{p}_n = J_{\gamma_nA}(w_n-\gamma_nUBw_n).
\end{equation}
Then $\overline{\pp}_n$ is $\FF_n$-measurable 
since $ J_{\gamma_n\AAA}\circ(\Id-\gamma_nUB)$ is continuous.
Therefore, by~\eqref{eq:co4} and (b1)
\begin{alignat}{2}
\label{e:ff1}
(\forall n\in\NN)\quad\|w_n-\overline{p}_n \|_{V}^2
&= \E[\|w_n-\overline{p}_n \|_{V}^2|\FF_n]\notag\\
&\leq 2\E[\|p_n-w_n \|_{V}^2|\FF_n]+2\E[\|\gamma_nU(\rr_n-Bw_n) \|_{V}^2|\FF_n] \to 0.
\end{alignat}

\ref{t:1i}:
Let $\omega\in{\Omega}$, and let $\overline{z}\in\mathcal{P}$ be a weak cluster point of $(x_n(\omega))_{n\in\NN}$.
Then, there exists a subsequence $(x_{k_n}(\omega))_{n\in\NN}$ which converges weakly to $\overline{z}$. It follows from our assumption that
$ (w_{k_n}(\omega))_{n\in\NN}$ converges weakly to $\overline{z}$.
 By \eqref{e:ff1}, $(\overline{p}_{k_n}(\omega))_{n\in\NN}$ converges weakly to $\overline{z}$. On the other hand, since 
$B$  is maximally monotone  and its graph is therefore 
sequentially closed in $\HH^{\text{weak}}\times\HH^{\text{strong}}$
\cite[Proposition~20.33(ii)]{livre1}, by \eqref{eq:co3},
$Bx = B\overline{z}$. 
\eqref{eq:rewr}, we have
\begin{equation}\label{eq:co7}
\frac{U^{-1}(w_{k_n}(\omega)-\overline{p}_{k_n}(\omega))}{\gamma_{k_n}}-Bw_{k_n}(\omega) \in Ap_{k_n}(\omega),
\end{equation}
and hence using the sequential closedness of $\gra A$ in
$\HH^{\text{weak}}\times\HH^{\text{strong}}$
\cite[Proposition~20.33(ii)]{livre1},
we get $-B\overline{z} \in A\overline{z}$ or equivalently, 
$\overline{z}\in\zer(A+B)=\mathcal{P}$. Therefore, for every $\omega\in\widehat{\Omega}$, 
every weak cluster point of $(x_n(\omega))_{n\in\NN}$
is in $\mathcal{P}$ which is a non-empty closed convex \cite[Proposition~23.39]{livre1}. 
Recalling~\eqref{e:ccc2} and applying Lemma \ref{p:fejer}\ref{p:fejeriv}, we derive that 
$(x_{n})_{n\in\NN}$ converges weakly to a  $\mathcal{P}$-valued random variable $\overline{x}$.

\ref{t:1ia}: Since $U$  is strongly positive, there exists a positive constant $\chi$ such that 
$(\forall y\in\HH)\; \scal{y}{Uy}\geq\chi \|y\|^2$. Therefore, we derive from 
\eqref{coco1} that 
\begin{equation}
(\forall z\in\HH)(\forall y\in\HH)\quad \| Bz-By\| \leq (\beta\chi)^{-1}\|z-y\|,
\end{equation}
which implies that 
\begin{alignat}{2}
(\forall n\in\NN)\quad \| Bx_n-Bw_n\| &\leq (\beta\chi)^{-1}\|x_n-w_n\|\notag\\
&= (\beta\chi)^{-1}\alpha_n\|x_n-x_{n-1}\|
&\to 0 \quad \text{by (c1)}.
\end{alignat}
Now, using \eqref{eq:co3}, we obtain $Bx_n\to B\overline{x}$.

\ref{t:1ii}: This conclusion follows from the definition of demiregular operator
and \ref{t:1ia}.
\end{proof}

\begin{corollary} 
Let $K$ be a strictly positive integer, let $\beta\in\RPP$,
let $\HH_1,\ldots,\HH_{K} $ be real Hilbert spaces. 
For every $i\in\{ 1,\ldots,K \}$, let $U_i \in \BL(\HH_i)$ be self-adjoint and  strongly positive,
let $A_i\colon \HH_i\to 2^{\HH_i }$ be maximally monotone, 
let $B_i\colon \HH_1\times\ldots\times\HH_{K}\to\HH_i$ such that
for every  
$\xx=(x_i)_{1\leq i\leq K}$ and $\yy= (y_i)_{1\leq i\leq K}$ in $ \HH_1\times\ldots\times\HH_{K}$,
\begin{alignat}{2}
\label{coco} 
\sum_{i=1}^{K} \scal{x_i-y_i}{B_i\xx - B_i\yy} \geq \beta 
\sum_{i=1}^{K} \scal{B_i\xx - B_i\yy }{U_i(B_i\xx - B_i\yy )}.
\end{alignat}
Suppose that the set $\mathcal{P}$ of  all points
$\overline{\xx} =(\overline{x}_1,\ldots, \overline{x}_{K}) $ in 
$\HH_1\times\ldots\times\HH_{K}$  such that
 \begin{equation}
\label{e:pris}
  \begin{cases}
0\in A_1\overline{x}_1 
+ B_1\overline{\xx} \\
  \vdots\quad  \\
0\in A_{K}\overline{x}_{K}
 + B_{K}\overline{\xx}
\end{cases}
\end{equation}
is non-empty. 
Let $\varepsilon\in\left]0,\min\{1,\beta\}\right[$,
let $(\gamma_n)_{n\in\NN}$ be a sequence in 
$\left[\varepsilon,(2-\varepsilon)\beta\right]$, let $(\lambda_n)_{n\in\NN}$ be a sequence in  $\left[\varepsilon,1\right]$,
and  let $(\alpha_n)_{n\in\NN}$ be a sequence in  $\left[0,1-\varepsilon\right]$.
For every $i\in\{ 1,\ldots,K \}$,
let $(r_{i,n})_{n\in\NN}$  be  a $\HH_i$-valued, square integrable random process,
let $x_{i,0}$ be  a $\HH_i$-valued, squared integrable random variable and set $x_{i,-1}=x_{i,0}$. Furthermore,
set 
\begin{equation}
\label{e:main1s}
(\forall n\in \NN)\quad
\begin{array}{l}
\left\lfloor
\begin{array}{l} 
\operatorname{For}\;i=1,\ldots, K\\
w_{i,n} = x_{i,n} + \alpha_n(x_{i,n}-x_{i,n-1})\\
z_{i,n}=w_{i,n}-  \gamma_n U_i r_{i,n} \\
 p_{i,n}= J_{\gamma_n U_{i} A_{i}}(z_{i,n})\\
 x_{i,n+1}=x_{i,n}+\lambda_{n}(p_{i,n}-x_{i,n}).\\
\end{array} 
\right.\\[2mm]
\end{array}
\end{equation}
Then Problem~\ref{e:pris} and Algorithm~\ref{e:main1s}  are special cases of Problem~\ref{e:pri} and Algorithm~\ref{e:main1s} respectively. 
\end{corollary}
\begin{proof}
Let $\HHH$ be the Hilbert direct sum $\HH_1\oplus\ldots\oplus\HH_{K}$ 
with the scalar product and the norm 
defined respectively by 
\begin{equation}
\label{scla}
 \pscal{\cdot}{\cdot}\colon (\xx,\yy)\mapsto 
\sum_{i=1}^{K} \scal{x_i}{y_i} \quad \text{and}\quad
\nnm\cdot\nnm^2\colon \xx\mapsto \pscal{\xx}{\xx},
\end{equation}
where we denote by 
$\xx =(x_i)_{1\leq i\leq K}$ and $\yy =(y_i)_{1\leq i\leq K}$ 
the generic elements in $\HHH$.
Set
\begin{equation}
\label{e:acl}
\begin{cases}
 \AAA\colon \HHH \to 2^{\HHH}\colon \xx \mapsto (A_ix_i)_{1\leq i\leq K},\\
\BBB \colon \HHH\to \HHH\colon \xx \mapsto  (B_i\xx)_{1\leq i\leq K},\\
\UU\colon \HHH\to \HHH\colon \xx \mapsto  (U_ix_i)_{1\leq i\leq K}.
\end{cases}
\end{equation}
Then $\mathcal{P}=\zer(\AAA+\BBB)$, $\AAA$ is maximally monotone by \cite[Proposition 20.23]{livre1}. 
Since $\UU$ is  self-adjoint and strongly positive, $\UU\AAA$ is also maximally monotone 
by Lemma \ref{l:maxmon45},
and by \cite[Proposition 23.16]{livre1}, its resolvent is
\begin{equation}
\label{eq:res} 
(\forall \xx\in\HHH)\; J_{\UU\AAA}\xx = (J_{U_iA_i}x_i)_{1\leq i\leq K}.
\end{equation}
Moreover,  in view of  \eqref{scla}, condition \eqref{coco} can be written as 
\begin{alignat}{2}
\label{coco1} 
\pscal{\xx-\yy}{\BBB\xx - \BBB\yy} \geq \beta 
\nnm\BBB\xx - \BBB\yy \nnm_{\UU}^2,
\end{alignat}
which shows that  $\BBB$ is monotone and continuous, and hence 
maximally monotone \cite[Corollary  20.25]{livre1}.
We define
\begin{equation}
\label{e:2012-05-01}
\begin{cases}
\zz_n=(z_{1,n},\ldots, z_{K,n}),\\
\xx_n=(x_{1,n},\ldots, x_{K,n}),\\
\ww_n=(w_{1,n},\ldots, w_{K,n}),\\
\pp_n=(p_{1,n},\ldots, p_{K,n}),\\
\rr_n=(r_{1,n},\ldots, r_{K,n}),\\
\end{cases}
\end{equation}
and we get
\begin{equation}
(\forall n\in\NN)\; 
\quad\FF_n = \sigma(\xx_0,\ldots,\xx_n).
\end{equation}
Moreover, in view of \eqref{e:acl} and \eqref{e:2012-05-01}, conditions
 \ref{cond:one1}, 
 \ref{cond:two2}, and \ref{cond:three3} can be rewritten as 
\begin{enumerate}
\item[(a1)]\label{cond:one11} For every $n\in\NN, \E\left[\rr_{n}| 
\FF_n\right] = \BB\ww_n$.
\item [(b1)] \label{cond:two22}$\sum_{n\in\NN}\gamma_{n}^2
\E\left[\nnm\rr_{n}- \BB\ww_n\nnm^2\; | \FF_n\right] < +\infty$.
\item[(c1)]\label{cond:three33} $\sup_{n\in\NN}\nnm \xx_{n}-\xx_{n-1}\nnm < \infty $ a.s. and $\sum_{n\in\NN}\alpha_n < +\infty$.
\end{enumerate}
Now, using  \eqref{eq:res} and \eqref{e:2012-05-01}, we can rewrite the algorithm  \eqref{e:main1*} as 
\begin{equation}
\label{eq:rewr}
(\forall n\in \NN)\quad
\begin{array}{l}
\left\lfloor
\begin{array}{l} 
\ww_{n}=\xx_{n}+\alpha_n  (\xx_n-\xx_{n-1}) \\
\zz_n = \ww_n-\gamma_n\UU \rr_n\\
 \pp_{n}= J_{\gamma_n\UU\AAA}(\zz_{n})\\
 \xx_{n+1}=\xx_{n}+\lambda_{n} (\pp_{n}-\xx_{n}).\\
\end{array} 
\right.\\[2mm]
\end{array}
\end{equation}

\end{proof}

\begin{remark}\label{r:1}
 Here are some comments concerning the demiregularity notion and the cocoercivity of $B$.
\begin{enumerate}
\item
\label{r:1zero}
Demiregularity is a general notion that captures several properties typically used to establish strong convergence 
of iterative algorithms. See \cite{plc2010} for a discussion and special cases.
\item 
\label{r:1i}
The condition \eqref{coco} is equivalent to the cocoercivity of $\sqrt{\UU}\BB\sqrt{\UU}$ which is weaker 
than the cocoercivity of $\BB$ and this condition was first considered in \cite{icip14}.
\item \label{r:1ii}
If $\BB$ is $\beta_0$-cocoercive, condition \eqref{coco}
is satisfied with $\beta = \beta_0/\|\UU\|$. Indeed, we have 
\begin{equation}
(\forall \xx\in\HHH)\quad \pscal{\xx}{\UU\xx} \leq \|\UU\| \|\xx\|^2.
\end{equation}
Therefore 
\begin{equation}
(\forall \xx\in\HHH)(\forall \yy\in\HHH)\quad \pscal{\xx-\yy}{\BB\xx-\BB\yy} \geq \beta_0\nnm \BB\xx-\BB\yy\nnm^2 
\geq \beta\nnm \BB\xx-\BB\yy\nnm_{\UU}^2.
\end{equation}
\end{enumerate}
\end{remark}

\begin{remark}
 Here are some connections to existing work.
\begin{enumerate}
\item The system of inclusions \ref{e:pris} was first studied  in \cite{plc2010}, in the special case $ (\forall i \in \{1,\ldots, K\})\; U_i = \Id$. 
Morever, in the same paper,  in the deterministic setting, a forward-backward splitting method \cite{siam05,mercier79} in a 
suitable product space was proposed for solving it. Furthermore, when in Problem \ref{prob1}
$(\forall n\in\NN)\; \rr_{n} = Bx_{n}$, the proposed algorithm 
reduces to the inertial forward-backward algorithm proposed in \cite{Dirk} and, in this case,
the weak convergence was proved in \cite{Dirk} where the condition \ref{cond:three3} in Theorem \ref{t:1} is replaced by
 the weaker condition that  $(\alpha_n\|x_{n}-x_{n-1}\|^2)_{n\in\NN}$ is summable.
\item If  $(\forall n\in\NN)\; \alpha_n=0$, 
the proposed method reduces to a stochastic forward-backward algorithm. 
In this case, almost sure convergence of the algorithm~\eqref{e:main1*} was proved 
in \cite{LSB14b} under the additional assumption that $B_1$ is uniformly monotone, 
and under some weaker conditions on the stochastic errors.  
\item If $(\forall n\in\NN)\;\alpha_n=0$,
 almost sure convergence of the algorithm ~\eqref{e:main1*} 
  for solving Problem~\ref{prob1} was proved 
in \cite{plc14} under the stronger condition  
that 
\begin{equation}
\sum_{n\in\NN}
\E[\|\rr_{n}- Bx_{n}\| | (x_{0}, \ldots, x_{n})] < +\infty.
\end{equation} 
\end{enumerate} 
\end{remark}
\section{Applications to composite monotone inclusions involving parallel sum}
\label{s:cop}
In this section, we focus on a  structured system of monotone inclusions which covers
a wide class of monotone inclusions involving cocoercive operators in the literature, see
\cite{plc04, icip14, optim2, pesquet14, Raguet11, jota1, Tseng91} and the references therein. 
The contribution of the section is twofold: on the one hand we will show that it is possible to prove 
convergence of many existing algorithms even in the presence of stochastic perturbations. 
On the other hand, we will derive two new classes of stochastic inertial primal-dual 
splitting methods based on different choices of the preconditioning operators.  
We remark that Algorithm~\eqref{e:Algomain3b}
is new even in the deterministic setting. 
\begin{problem}
\label{probbb}
Let $m$ and $s$ be strictly positive integers, let $\nu_0$ and $\mu_0$
be in $\left]0,+\infty\right[$.
For every $i\in \{1,\ldots,m\}$, 
let $(\KK_i,\scal{\cdot}{\cdot})$ be a real Hilbert space,
let $z_i\in\KK_i$,   
let $A_{i}\colon\KK_i \to 2^{\KK_i}$
be maximally monotone, let $V_i\in\BL(\KK_i)$ be self-adjoint and 
strongly positive,
let $C_i\colon \KK_1\times\ldots\times\KK_m\to\KK_i$ be such that 
for every 
$\xx=(x_i)_{1\leq i\leq m}$ and 
$\yy= (y_i)_{1\leq i\leq m}$ in $ \KK_1\times\ldots\times\KK_m$,
\begin{alignat}{2}
\label{recoco}
\sum_{i=1}^m \scal{x_i-y_i}{C_i\xx - C_i\yy} \geq \nu_0 
\sum_{i=1}^m \|C_i\xx - C_i\yy\|_{V_i}^2.
\end{alignat}
For every $k\in\{1,\ldots, s\}$,
let $(\GG_k,\scal{\cdot}{\cdot})$ be a real Hilbert space, 
let $B_{k}\colon \GG_k\to 2^{\GG_k}$ 
be maximally monotone,
let $r_k \in \GG_k$, let $W_k\in\BL(\GG_k)$ be self-adjoint and 
strongly positive,
let $D_{k}\colon \GG_k\to 2^{\GG_k}$ be maximally monotone 
and suppose that $D_{k}^{-1}$ is single-valued and such that, for every $v_k\in\GG_k$
 and $w_k\in\GG_k$,
\begin{alignat}{2}
\label{recocoo}
 \sum_{k=1}^s\scal{v_k-w_k}{D_{k}^{-1}v_k - D_{k}^{-1}w_k} \geq \mu_0 
\sum_{k=1}^s \| D_{k}^{-1}v_k - D_{k}^{-1}w_k\|_{W_k}^2 .
\end{alignat}
 For every $i \in \{1,\ldots, m\}$ and every $k\in\{1,\ldots,s\}$,
let $L_{k,i}\colon\KK_i \to\GG_k$ 
be a  bounded linear operator.
Suppose that the set $\mathcal{P}$ of all point 
$\overline{\xx} =  (\overline{x}_1,\ldots,\overline{x}_m)$ in $ \KK_1\times \ldots\times \KK_m$ such that
\begin{alignat}{2}\label{reprimalin}
\begin{cases} 
z_1\in A_1\overline{x}_1+ \displaystyle\sum_{k=1}^s L_{k,1}^*
\bigg(( D_k\;\vuo\; B_{k})\bigg(\displaystyle\sum_{i=1}^m L_{k,i}\overline{x}_i-r_k\bigg)\bigg) 
+ C_1\overline{\xx}\\
\vdots\quad\\
z_m\in A_m\overline{x}_m+ \displaystyle\sum_{k=1}^s L_{k,m}^*
\bigg(( D_k\;\vuo\; B_{k})\bigg(\displaystyle\sum_{i=1}^m L_{k,i}\overline{x}_i-r_k\bigg)\bigg) 
+ C_m\overline{\xx}\\
\end{cases}
\end{alignat}
is non-empty. 
Denote by  $\mathcal{D}$  the set of all solutions  
$\overline{\boldsymbol{v}}=(\overline{v}_1,\ldots,\overline{v}_s) \in \GG_1\times\ldots\times\GG_s$ to the dual inclusion
 \begin{alignat}{2}\label{redualin}
&\big(\exists \xx = (x_i)_{1\leq i\leq m}\in (\KK_i)_{1\leq i\leq m}\big)\quad\notag\\
&\begin{cases}
z_1-\displaystyle\sum_{k=1}^s L_{k,1}^*\overline{v}_k\in A_1x_1 + C_1\xx \\
  \vdots\quad\\
z_m-\displaystyle\sum_{k=1}^s L_{k,m}^*\overline{v}_k\in A_mx_m + C_m\xx, 
\end{cases}
\text{and}\quad
\begin{cases}
\displaystyle\sum_{i=1}^mL_{1,i}x_i-r_1 \in B_{1}^{-1}\overline{v}_1 + D_{1}^{-1}\overline{v}_1\\
  \vdots\quad\\
\displaystyle\sum_{i=1}^mL_{s,i}x_i-r_s \in B_{s}^{-1}\overline{v}_s + D_{s}^{-1}\overline{v}_s.
\end{cases}
\end{alignat}
The problem is to find a point in  $\mathcal{P}\times\mathcal{D}$.
\end{problem}

\begin{remark}
As noted in \cite{jota1}, in Problem~\ref{probbb}, the variables are coupled in two different ways.
The first one is the smooth coupling induced by $(C_i)_{1\leq i\leq m}$. 
The second one is the non-smooth coupling involved in the parallel sums in the second terms in 
\eqref{reprimalin}. 
\end{remark}

Note that, since we assume that $\mathcal{P}$ is nonempty,  $\mathcal{D}$ is  nonempty  as well. 
Let us introduce the Hilbert direct sums 
\begin{equation}
\label{e:HeK}
\KKK = \KK_1\oplus\ldots\oplus\KK_m, \quad \GGG = \GG_1\oplus\ldots\oplus\GG_s, 
\quad \text{and}\quad \HHH = \KKK\oplus\GGG,
\end{equation}
endowed with the scalar product and the norm  defined as in \eqref{scla}. 
With a slight abuse of notation, in all spaces, 
the scalar products and norms are denoted as $\scal{\cdot}{\cdot}$
and $\|\cdot\|$, respectively.
We denote by 
$\xx = (x_i)_{1\leq i\leq m}$, $\yy = (y_i)_{1\leq i\leq m}$ 
the generic elements in $\KKK$, and by 
$\vv = (v_k)_{1\leq k\leq s}$, $\ww = (w_k)_{1\leq k\leq s}$ 
the generic elements in  $\GGG$.
The generic elements in  $\HHH$ will be denoted by $ \xk$ and $\yk$.
We also consider the linear operators
\begin{equation}
\label{eq:VeW}
\begin{cases}
\LL\colon\KKK\to\GGG
\colon\xx\mapsto\big(\sum_{i=1}^mL_{k,i}x_i\big)_{1\leq k\leq s}\\
\VV\colon\KKK\to\KKK\colon\xx
\mapsto \big( V_ix_i\big)_{1\leq i\leq m}\\
\WW\colon\GGG \to \GGG \colon \vv
\mapsto \big( W_kv_k\big)_{1\leq k\leq s}.\\
\end{cases}
\end{equation}
We first need the following lemma which follows 
from \cite[Lemma 4.3(i) and Lemma 4.9(i)]{pesquet14}.

\begin{lemma}
\label{l:1a}
In the setting of Problem~\ref{probbb}, let $\LL$, $\VV$, and $\WW$ be defined as in \eqref{eq:VeW}. 
Suppose that $\|\sqrt{\WW}\LL\sqrt{\VV}\| < 1$ and set
\begin{equation}
\label{e:UeT}
\begin{cases}
\UU^{\prime}\colon \HHH\to \HHH\colon (\xx,\vv)\mapsto 
\big(\VV^{-1}\xx- \LL^*\vv, \WW^{-1}\vv-\LL\xx \big)\\
\TT\colon \HHH\to \HHH\colon (\xx,\vv)\mapsto  (\VV\xx, (\WW^{-1}- \LL\VV\LL^* )^{-1}\vv).\\
\end{cases}
\end{equation}
Then $\UU^{\prime}$ and $\TT$ are self-adjoint and strongly positive, with
\begin{equation}
(\forall \xk \in \HHH)\quad \scal{\UU^{\prime}\xk}{\xk} \geq \frac12(1-\|\sqrt{\WW}\LL\sqrt{\VV}\|^2) 
\min\{\|\VV^{-1}\|,\|\WW^{-1}\|\}\|\xk\|^2,
\end{equation}
and 
\begin{equation}
(\forall \xk \in \HHH)\quad \scal{\TT\xk}{\xk} \geq
\min\{\|\VV^{-1}\|^{-1}, \| \WW^{-1}- \LL\VV\LL^*\|^{-1}\}\|\xk\|^2.
\end{equation}
In particular, $\UU^{\prime}$  is invertible, and its inverse $\UU=(\UU^\prime)^{-1}$ is self-adjoint and 
strongly positive. 
\end{lemma}
\begin{lemma}
\label{l:2a}
Consider the setting of Problem~\ref{probbb}, and define
\begin{equation}
\label{e:2a}
\begin{cases}
 \MM\colon \HHH\to 2^{\HHH}\colon (\xx,\vv)
\mapsto \big( (-z_i + A_ix_i)_{1\leq i\leq m}, (r_k+ B^{-1}_kv_k)_{1\leq k\leq s}\big) \\
 \SSS\colon \HHH\to 2^{\HHH}\colon (\xx,\vv)
\mapsto \big( (\sum_{k=1}^s L_{k,i}^*v_k)_{1\leq i\leq m}, (-\sum_{i=1}^mL_{k,i}x_i)_{1\leq k\leq s}\big) \\
\QQ\colon \HHH\to {\HHH}\colon (\xx,\vv)
\mapsto \big( (C_i\xx )_{1\leq i\leq m}, ( D^{-1}_kv_k)_{1\leq k\leq s}\big). 
\end{cases}
\end{equation} 
Then the following hold. 
\begin{enumerate}
\item
\label{l:2ai}
 Problem \ref{probbb} is a special case of Problem \ref{prob1} with 
$\HH=\HHH, A= \MM+\SSS$, $B = \QQ$ and $U\in\{\UU,\TT\}$. 
\item
\label{l:2aii}
$(\emp \not=\mathcal{P}) \Rightarrow (\emp\not=\zer(\MM + \SSS + \QQ)\subset \mathcal{P}\times\mathcal{D}) $.
\end{enumerate} 
\end{lemma}
\begin{proof}
\ref{l:2ai}: We note that the operators $\VV$ and $\WW$, defined in equation~\eqref{eq:VeW}, are self-adjoint and 
strongly positive on $\KKK$ and $\GGG$, respectively.
Set
\begin{equation}
\label{e:acl1}
\begin{cases}
 \AAA\colon \KKK \to 2^{\KKK}\colon \xx \mapsto \cart_{i=1}^m A_ix_i\\
\BB\colon \GGG\to 2^{\GGG}\colon \vv \mapsto \cart_{k=1}^s B_kv_k\\
\CCC \colon \KKK\to \KKK\colon \xx \mapsto  (C_i\xx)_{1\leq i\leq m}\\ 
\DD\colon \GGG\to 2^{\GGG}\colon \vv \mapsto  (D_kv_k)_{1\leq k\leq s}\\
\zz =(z_1,\ldots, z_m)\\ 
 \rr =(r_1,\ldots, r_s).
\end{cases}
\end{equation}
Then, it follows from \eqref{recoco} that 
\begin{equation}
\label{e:coso}
 (\forall \xx\in\KKK)(\forall\yy\in \KKK)\quad 
\scal{\xx-\yy}{\CCC\xx-\CCC\yy}\geq \nu_0 \|\CCC\xx-\CCC\yy \|_{\VV}^2,
\end{equation}
and from \eqref{recocoo} that 
\begin{equation}
\label{e:coso1}
 (\forall \vv\in\GGG)(\forall\ww\in \GGG)\quad 
\scal{\vv-\ww}{\DD^{-1}\vv-\DD^{-1}\ww}\geq \mu_0 \|\DD^{-1}\vv-\DD^{-1}\ww \|_{\WW}^2.
\end{equation}
In view of Remark \ref{r:1}\ref{r:1i}, $\sqrt{\VV}\CCC\sqrt{\VV}$ and $\sqrt{\WW}\DD^{-1}\sqrt{\WW}$ are, 
respectively, $\nu_0$ and $\mu_0$ cocoercive. Therefore,  by \cite[Lemma 4.3(ii)]{pesquet14},  
we obtain, for every $\xi\in\left]0,+\infty\right[$,
\begin{equation}
\label{ee1}
(\forall \xk\in\HHH)(\forall \yk \in\HHH)\quad 
\scal{\xk-\yk}{\QQ\xk-\QQ\yk}\geq \beta_{\xi}\|\QQ\xk-\QQ\yk \|_{\UU}^2,
\end{equation}
where $\beta_{\xi}$ is defined by
\begin{equation}
\label{betaxi}
\beta_{\xi} = (1-\|\sqrt{\WW}\LL\sqrt{\VV}\|^2)
\min\{\nu_0(1+\xi \|\sqrt{\WW}\LL\sqrt{\VV}\|)^{-1}, \mu_0(1+\xi^{-1} \|\sqrt{\WW}\LL\sqrt{\VV}\|)^{-1}\}.
\end{equation}
By \cite[Lemma 4.9(ii)]{pesquet14},  
we obtain 
\begin{equation}
\label{ee2}
(\forall \xk\in\HHH)(\forall \yk \in\HHH)\quad 
\scal{\xk-\yk}{\QQ\xk-\QQ\yk}\geq \beta\|\QQ\xk-\QQ\yk \|_{\TT}^2,
\end{equation}
where  $\beta $ is defined by
\begin{equation}
\label{beta}
\beta = \min\{\nu_0,\mu_0(1-\|\sqrt{\WW}\LL\sqrt{\VV}\|^2)\}.
\end{equation}
 Since both $\UU$ and $\TT$ are strongly positive by Lemma \ref{l:1a}, 
either \eqref{ee1} or \eqref{ee2} implies the cocoercivity of $\QQ$
and hence
 $\QQ$ is maximally monotone  
 \cite[Corollary 20.25]{livre1}. 
Moreover, it follows from \cite[Proposition 20.23]{livre1} 
that $\AAA$ and  $\BB$  are maximally monotone. 
Let us define
\[
 \LL^*\colon \GGG\to\KKK\colon \vv \mapsto 
\bigg(\sum_{k=1}^s L_{k,i}^*v_k\bigg)_{1\leq i\leq m},
\]
and consider the following  inclusion in the space  $\HHH$,
\begin{equation}
\label{e:pr}
\big(\zz-\LL^*\overline{\vv}, \LL\overline{\xx}-\rr\big)\in
\Big( (\AAA+\CCC)\overline{\xx}, (\BB^{-1} +\DD^{-1})\overline{\vv}\Big) .
\end{equation}
Note that we can rewrite $\MM$, $\SSS$ and $\QQ$ as follows
\begin{equation}
\label{msq}
\begin{cases}
\MM\colon \HHH\to 2^{\HHH}\colon (\xx,\vv)\mapsto (-\zz+\AAA\xx,\rr+\BB^{-1}\vv)\\
\SSS\colon \HHH\to\HHH\colon (\xx,\vv) \mapsto (\LL^*\vv,-\LL\xx)\\
\QQ \colon \HHH\to \HHH\colon (\xx,\vv)\mapsto (\CCC\xx, \DD^{-1}\vv).\\
\end{cases}
\end{equation}
Thus, $\MM$ and $\SSS$ are maximally monotone.  Moreover,
$\MM+\SSS$ is maximally monotone since $\SSS$ is maximally monotone 
 and single-valued \cite[Corollary 24.4]{livre1}. 

\ref{l:2aii}: Note that $\mathcal{P}\neq \emp$ implies that 
\begin{equation}
\label{nonem}
 \zer(\MM+\SSS+\QQ) \neq \emp.
\end{equation}
Furthermore, the problem \eqref{e:pr} reduces to find a random vector 
$\zer(\MM+\SSS+\QQ)$-valued almost surely. Next, let $(\xx,\vv)$ be a solution to
\eqref{e:pr}. Then, by removing $\vv$ from  \eqref{e:pr}, we obtain 
\begin{equation}
\zz \in (\AAA+\CCC)\xx + \LL^*(\BB\vuo\DD)(\LL\xx-\rr),
\end{equation}
which implies that $\xx \in\mathcal{P}$. By the same way, removing 
$\xx$ from  \eqref{e:pr}, we obtain
\begin{equation}
-\rr \in -\LL(\AAA+\CCC)^{-1}(\zz-\LL^*\vv) + \BB^{-1}\vv +\DD^{-1}\vv.
\end{equation}
Therefore, there exists $\overline{\xx}\in\HHH$ such that 
\begin{equation}
\overline{\xx} \in (\AAA+\CCC)^{-1}(\zz-\LL^*\vv)
\quad \text{and}\quad  \LL\overline{\xx}-\rr \in   \BB^{-1}\vv +\DD^{-1}\vv,
\end{equation}
which implies that $\vv\in\mathcal{D}$. 
To sum up, $\zer(\MM+\SSS+\QQ)\subset \mathcal{P}\times\mathcal{D}$.
\end{proof}

\begin{remark} Proceeding as in \cite[Remark 4.4(i)]{pesquet14}, if we maximize 
$\beta_\xi$ with respect to $\xi$, we get that $\bf{Q}$ satisfies 
\begin{equation}
\label{e:cocopt}
(\forall \xk\in\HHH)(\forall \yk \in\HHH)\quad 
\scal{\xk-\yk}{\QQ\xk-\QQ\yk}\geq \beta_{\hat{\xi}}\|\QQ\xk-\QQ\yk \|_{\UU}^2
\end{equation}
where
\begin{equation}
\label{e:optxi}
\hat{\xi}=\frac{\nu_0-\mu_0+\sqrt{(\mu_0-\nu_0)^2+4\|\sqrt{\WW}\LL\sqrt{\VV}\|^{2}\nu_0\mu_0}}{2\mu_0\|\sqrt{\WW}\LL\sqrt{\VV}\|}
\end{equation}

\end{remark}
\subsection{A first class of stochastic inertial primal-dual splitting methods}
Our first class of  stochastic  primal-dual splitting algorithm for solving Problem \ref{probbb} which corresponds 
to the choice of $U= \UU$ in Lemma \ref{l:2a}.
\begin{algorithm}
\label{mainal}
Let $\hat{\xi}\in \left]0,+\infty\right[$ be defined by \eqref{e:optxi}, 
$\widehat{\beta}=\beta_{\hat{\xi}}$ be defined according to \eqref{betaxi}, 
let $\varepsilon \in \left]0,\min\{1,\widehat{\beta}\}\right[$,
let $(\lambda_n)_{n\in\NN}$ be a sequence 
in $\left[\varepsilon,1\right]$, and let $(\alpha_n)_{n\in\NN}$ be a sequence in $\left[0,1-\varepsilon\right]$.
For every $i\in \{1,\ldots, m\}$, 
let $(\rh_{i,n})_{n\in\NN}$ be a $\KK_i$-valued, squared integrable random process, and let
$x_{i,0}$  be a $\KK_i$-valued, squared integrable random vector and set $x_{i,-1}=x_{i,0}$.  
For every $k\in \{1,\ldots,s\}$, 
let 
$(\og_{i,n})_{n\in\NN}$ be a $\GG_i$-valued, squared integrable random process, and 
$v_{i,0}$  be a $\GG_i$-valued, squared integrable random vector and set $v_{i,-1}=v_{i,0}$.
Then, iterate, for every $n\in\NN$,
 \begin{equation}\label{e:Algomain}
 \begin{array}{|l}
(i)\;\operatorname{For}\;i=1,\ldots, m\\
\quad
\begin{array}{|l}
 c_{i,n} = x_{i,n} + \alpha_n(x_{i,n}- x_{i,n-1} )\\
\end{array}\\
(ii)\;\operatorname{For}\;k=1,\ldots, s\\
\quad
\begin{array}{|l}
 d_{k,n} = v_{k,n} + \alpha_n(v_{k,n}- v_{k,n-1} )\\
\end{array}\\
(iii)\;  \operatorname{For}\;i=1,\ldots, m\\
\quad
\begin{array}{|l}
 1.\; t_{i,n} =\sum_{k=1}^sL_{k,i}^*d_{k,n}+\rh_{i,n}\\
2.\; p_{i,n}:=J_{V_{i}A_{i}}\big(c_{i,n}-  V_{i}
(t_{i,n} -z_i)\big)\\
3.\; y_{i,n}:=2p_{i,n}-c_{i,n}\\
4.\; x_{i,n+1}:=x_{i,n}+\lambda_{n}(p_{i,n}-x_{i,n})\\
\end{array}\quad\\
(iv)\; \operatorname{For}\;k=1,\ldots, s\\
\quad
\begin{array}{|l}
1.\; u_{k,n} =\sum_{i=1}^m L_{k,i}y_{i,n}- \og_{k,n}\\
2.\; q_{k,n}:=J_{W_{k}B^{-1}_{k}}
\big(d_{k,n}+W_{k}\big(u_{k,n}-r_k\big)\big)\\
3.\; v_{k,n+1}:=v_{k,n}+\lambda_{n}(q_{k,n}-v_{k,n}).\\
\end{array}
\end{array}
\end{equation} 
\end{algorithm}

\begin{theorem} 
\label{t:2}
Consider Algorithm~\ref{mainal} and suppose that 
\begin{equation}
\label{e:2f9h79p}
\widehat{\beta} > 1/2. 
 \end{equation}
Set
\begin{equation}
(\forall n\in\NN)\quad
\begin{cases}
 \xk_n:=(x_{1,n},\ldots, x_{m,n},v_{1,n},\ldots,v_{s,n})\\
\rk_n = (\rh_{1,n},\ldots, \rh_{m,n},\og_{1,n},\ldots,\og_{s,n})\\
\FF_n = \sigma(\xk_0,\ldots,\xk_n)\\
\end{cases}
\end{equation}
and suppose that the following conditions are satisfied:
\begin{enumerate}
\item\label{cond:onei} 
$ (\forall n\in \NN)\;
\E[\rk_{n}| \FF_n] = \big((C_i(c_{1,n}\ldots, c_{m,n}))_{1\leq i\leq m}, D_{1}^{-1}d_{1,n},\ldots, D_{s}^{-1}d_{s,n}\big)$ a.s.
\item \label{cond:twoii}
$\sum_{n\in\NN}
\E[ \sum_{i=1}^m\|\rh_{i,n}-C_i(c_{1,n}\ldots, c_{m,n})\|^2
+ \sum_{k=1}^s\|  \og_{k,n}- D_{k}^{-1}d_{k,n}\|^2   
| \FF_n] < +\infty$ a.s.
\item\label{cond:twoiii}
 $\sup_{n\in\NN}\sum_{i=1}^m\| x_{i,n}-x_{i,n-1}\|^2 < \infty $ a.s.
and  $\sup_{n\in\NN}\sum_{k=1}^s\| v_{k,n}-v_{k,n-1}\|^2< \infty $ a.s.,
and $\sum_{n\in\NN}\alpha_n < +\infty $.
\end{enumerate}
Then the following hold for some random vector 
$(\overline{x}_1,\ldots,\overline{x}_m, \overline{v}_1,\ldots,\overline{v}_s)$,
$\mathcal{P}\times\mathcal{D}$-valued a.s.
\begin{enumerate}
 \item
\label{t:2i} $(\forall i\in\{1,\ldots,m\})$\; $x_{i,n}\weakly \overline{x}_i$ and 
$(\forall k\in\{1,\ldots,s\})$\; $v_{k,n}\weakly\overline{v}_k$ a.s.
\item \label{t:2ii}
Suppose that the operator $(x_i)_{1\leq i\leq m} \mapsto (C_j(x_i)_{1\leq i\leq m})_{1\leq j\leq m}$ 
is demiregular at 
$(\overline{x}_1,\ldots,\overline{x}_m)$, then $(\forall i\in\{1,\ldots,m\})\; $
$x_{i,n}\to \overline{x}_i$ a.s. 
\item 
\label{t:2iii}
Suppose that there exists $j\in\{1,\ldots,m\}$ such that $D_{j}^{-1}$ is demiregular 
at 
$\overline{v}_j$, then $ 
v_{j,n}$
$\to \overline{v}_j$ a.s. 
\item
\label{t:2iv}
Suppose that there exists $j\in\{1,\ldots,m\}$ and an operator 
$C\colon\KK_j \to \KK_j$ such that 
$(\forall(x_i)_{1\leq i\leq m}\in(\KK_i)_{1\leq i\leq m})\; C_j(x_1,\ldots,x_m) = Cx_j$ and 
$C$ is demiregular at $\overline{x}_j$, then
$x_{j,n}\to \overline{x}_j$ a.s. 
\end{enumerate}
\end{theorem}
\begin{proof}
We first observe that \eqref{e:Algomain} is equivalent to 
\begin{equation}
\label{e:Algomainr}
\begin{array}{|l}
(i)\;\operatorname{For}\;i=1,\ldots, m\\
\quad
\begin{array}{|l}
 c_{i,n} = x_{i,n} + \alpha_n(x_{i,n}- x_{i,n-1} )\\
\end{array}\\
(ii)\;\operatorname{For}\;k=1,\ldots, s\\
\quad
\begin{array}{|l}
 d_{k,n} = v_{k,n} + \alpha_n(v_{k,n}- v_{k,n-1} )\\
\end{array}\\
(iii)\;\operatorname{For}\;i=1,\ldots, m\\
\quad
\begin{array}{|l}
1.\;V_{i}^{-1}(c_{i,n}- p_{i,n})-\sum_{k=1}^sL_{k,i}^*d_{k,n}-\rh_{i,n}\in
-z_i + A_ip_{i,n}\\
2.\;x_{i,n+1}=x_{i,n}+\lambda_{n}(p_{i,n}-x_{i,n})\\
\end{array}
\quad\\
(iv)\; \operatorname{For}\;k=1,\ldots, s\\
\quad
\begin{array}{|l}
1.\; W_{k}^{-1}(d_{k,n}-q_{k,n}) - \sum_{i=1}^m L_{k,i}(c_{i,n}- 
p_{i,n})- \og_{i,n}\in 
B_{k}^{-1}q_{k,n}
-\sum_{i=1}^m L_{i}p_{i,n}\\
2.\;v_{k,n+1}=v_{k,n}+\lambda_{n}(q_{k,n}-v_{k,n}).\\
\end{array}
\end{array}
\end{equation}
Upon setting
\begin{equation}
\label{ckyk}
(\forall n\in\NN)\quad
\begin{cases}
\ck_n = (c_{1,n},\ldots,c_{m,n},d_{1,n},\ldots,c_{s,n})\\
\yk_n = (p_{1,n},\ldots,p_{m,n},q_{1,n},\ldots,q_{s,n}),\\
\end{cases}
\end{equation}
we can rewrite \eqref{e:Algomainr} as the following 
\begin{equation}
\label{e:Algomainr}
(\forall n\in\NN)\quad
\begin{array}{|l}
1.\; \ck_{n} = \xk_{n} + \alpha_n(\xk_{n}- \xk_{n-1} )\\
2.\; \UU^{-1}(\ck_n-\yk_n)-\rk_n\in \MM\yk_n +\SSS \yk_n\\
3.\;\xk_{n+1}=\xk_{n}+\lambda_{n}(\yk_{n}-\xk_{n})\\
\end{array}
\Leftrightarrow 
\begin{array}{|l}
1.\; \ck_{n} = \xk_{n} + \alpha_n(\xk_{n}- \xk_{n-1} )\\
2.\; \yk_n = J_{\UU(\MM+\SSS)}(\ck_n-\UU\rk_n)\\
3.\;\xk_{n+1}=\xk_{n}+\lambda_{n}(\yk_{n}-\xk_{n}),\\
\end{array}
\end{equation}
which is a special instance of the iteration \eqref{e:main1*} with 
$(\forall n\in\NN)\;\gamma_n =1\in \left]\varepsilon,(2-\varepsilon)\widehat{\beta} \right] $.
We next see that our conditions can be rewritten in the space $\HHH$ as 
\begin{enumerate}
\item[(a1)]\label{cond:one} 
$ (\forall n\in \NN)\;
\E[\rk_{n}| \FF_n] = \QQ\ck_n.
$
\item [(b1)] \label{cond:two}
$\sum_{n\in\NN} \E[ \|  \rk_{n}-\QQ\ck_n\|^2| \FF_n] < +\infty.$
\item[(c1)] $\sup_{n\in\NN}\| \xk_{n}-\xk_{n-1}\|^{2} < \infty $ a.s.
and $\sum_{n\in\NN}\alpha_n < +\infty $.
\end{enumerate}
Therefore, every specific conditions in Algorithm \ref{a:al1} and Theorem \ref{t:1} are satisfied.

\ref{t:2i}: In view of Theorem \ref{t:1}\ref{t:1i}, $\xk_n\weakly (\overline{\xx},\overline{\vv})$ which is equivalent to 
$(\forall i\in\{1,\ldots,m\})\; x_{i,n}\weakly \overline{x}_i$ and $(\forall k\in\{1,\ldots,s\})\; v_{k,n}\weakly \overline{v}_k$.

\ref{t:2ii}$\&$ \ref{t:2iii}: By Theorem  \ref{t:1}\ref{t:1ia}, we have $\QQ\xk_n\to \QQ\overline{\xk}$ which is equivalent to 
\begin{equation}
\begin{cases}
 (\forall i\in\{1,\ldots,m\})\quad C_i(x_{1,n},\ldots,x_{m,n}) \to C_i(\overline{x}_{1},\ldots,\overline{x}_m)\\
 (\forall k\in\{1,\ldots,s\})\quad D_{k}^{-1}v_{k,n} \to D_{k}^{-1}\overline{v}_k.\\
\end{cases}
\end{equation}
Therefore, if the operator $(x_i)_{1\leq i\leq m} \mapsto (C_j(x_i)_{1\leq i\leq m})_{1\leq j\leq m}$ is demiregular at 
$(\overline{x}_1,\ldots,\overline{x}_m)$, we obtain $(\forall i\in\{1,\ldots,m\})\; x_{i,n}\to\overline{x}_i $. By the same 
season,  if there exists $j\in \{1,\ldots,m\}$ such that $D_{j}^{-1}$ is demiregular at $\overline{v}_j$, the 
$v_{j,n}\to\overline{v}_j $.

\ref{t:2iv} This conclusion follows from the definition of the demiregular operators by the same reason as in \ref{t:2iii}.
 \end{proof}

We next provide an application to the following minimization problem considered in \cite[Problem 5.1]{jota1}, where
several applications are discussed.

\begin{example}
\label{ex:prob2a} 
Let $m$ and $s$ be strictly positive integers.
For every $i\in\{1,\ldots,m\}$, 
let $\KK_i$ be a real Hilbert space,
let $z_i\in\KK_i$, let $f_{i}\in\Gamma_0(\KK_i)$.
For every $k\in\{1,\ldots, s\}$,
let $\GG_k$ be a real Hilbert space, 
let $r_k \in \GG_k$, let $\ell_{k}\in\Gamma_0(\GG_k)$ be a strongly convex function, 
let $g_{k}\in \Gamma_0(\GG_k)$. 
For every $i \in \{1,\ldots, m\}$ and every $k\in\{1,\ldots,s\}$,
let $L_{k,i}\colon\KK_i \to\GG_k$ 
be a  bounded linear operator.
Let $\varphi\colon \KK_1\times\ldots \times\KK_m \to \RR$ 
be a convex differentiable function with a Lipschitz continuous gradient.  
Suppose that there exists 
$\overline{\xx} =(\overline{x}_1,\ldots,\overline{x}_m)$ such that,
for every $ i\in\{1,\ldots,m\}$,
\begin{equation}
\label{e:bfi}
z_i\in  \partial f_i(\overline{x}_i)+ \sum_{k=1}^s L_{k,i}^*\circ
\Big(\partial \ell_k\;\vuo\; \partial g_{k}\Big)\circ\bigg(\sum_{j=1}^m L_{k,j}\overline{x}_j-r_k\bigg) + 
\nabla_{i}\varphi(\overline{\xx}),
\end{equation}
where $\nabla_i\varphi$ is the $i$-{{\small\em th}} component of the gradient $\nabla\varphi$,
and that the set $\mathcal{P}$ of solutions to the primal problem
\begin{alignat}{2} 
\label{primal2}
\underset{x_1\in\KK_1,\ldots, x_m\in\KK_m}{\text{minimize}}
\sum_{i=1}^{m}\big(f_i(x_i) - \scal{x_i}{z_i}\big)
+ \sum_{k=1}^s\big(\ell_k\;\vuo\; g_{k}\big)\bigg(\sum_{i=1}^m L_{k,i}x_i -r_k\bigg)\notag\\
\hfill + \varphi(x_1,\ldots,x_m),
\end{alignat}
is nonempty. Denote by $\mathcal{D}$ the set of solutions
to the dual problem  
\begin{alignat}{2}
\label{dual2} 
\underset{v_1\in\GG_1,\ldots, v_s\in\GG_s}{\text{minimize}}
&\bigg(\varphi^{*}\;\vuo\;\bigg(\sum_{i=1}^m f_{i}^*\bigg)\bigg)\bigg(\Big(z_i-\sum_{k=1}^s L_{k,i}^*v_k\Big)_{1\leq i\leq m} \bigg)\notag\\
&\hspace{4cm}
+\sum_{k=1}^{s}\bigg(\ell^{*}_k(v_k) + g_{k}^*(v_k) + \scal{v_k}{r_k}\bigg).
\end{alignat}
The problem is then to find a random vector $\mathcal{P}\times\mathcal{D}$-valued almost surely.
\end{example}

\begin{corollary} 
For every $i\in \{1,\ldots, m\}$, 
let $V_{i}\in\BL(\KK_i)$ be self-adjoint and strongly positive.
 Let $\nu_0$  be a strictly positive number such that 
  for every 
$\xx=(x_i)_{1\leq i\leq m}$ and 
$\yy= (y_i)_{1\leq i\leq m}$ in $ \KK_1\times\ldots\times\KK_m$,
\begin{alignat}{2}
\label{cos1}
\sum_{i=1}^m \scal{x_i-y_i}{\nabla_i\varphi(\xx) - \nabla_i\varphi(\yy)} \geq \nu_0 
\sum_{i=1}^m \| \nabla_i\varphi(\xx) - \nabla_i\varphi(\yy)\|_{V_i}^2.
\end{alignat}
For every $k\in \{1,\ldots, s\}$, 
let $W_{k}\in\BL(\GG_k)$ be self-adjoint and strongly positive.
 Let $\mu_0$  be a strictly positive number such that 
  for every 
$\vv=(v_k)_{1\leq k\leq s}$ and 
$\ww= (w_k)_{1\leq k\leq s}$ in $ \GG_1\times\ldots\times\GG_s$,
\begin{alignat}{2}
\label{cos2}
\sum_{k=1}^s \scal{v_k-w_k}{\nabla \ell_{k}^*(v_k) -\nabla \ell_{k}^*(w_k) } \geq \mu_0 
\sum_{k=1}^s \|\nabla \ell_{k}^*(v_k) -\nabla \ell_{k}^*(w_k) \|_{W_k}^2.
\end{alignat}
Let $\hat{\xi}\in \left]0,+\infty\right[$ be defined by \eqref{e:optxi}, 
$\widehat{\beta}=\beta_{\hat{\xi}}$ be defined according to \eqref{betaxi}, 
 let $\varepsilon \in \left]0,\min\{1,\widehat{\beta}\}\right[$,
let $(\lambda_n)_{n\in\NN}$ be a sequence 
in $\left[\varepsilon,1\right]$,
let $(\alpha_n)_{n\in\NN}$ be a sequence in $\left[0,1-\varepsilon\right]$.
For every $i\in \{1,\ldots, m\}$, 
 let $(\rh_{i,n})_{n\in\NN}$ be a $\KK_i$-valued, squared integrable random process, and let
$x_{i,0}$  be a $\KK_i$-valued, squared integrable random vector and set $x_{i,-1}=x_{i,0}$.  
For every $k\in \{1,\ldots,s\}$, 
let $(\og_{i,n})_{n\in\NN}$ be a $\GG_i$-valued, squared integrable random process, and let
$v_{i,0}$  be a $\GG_i$-valued, squared integrable random vector and set $v_{i,-1}=v_{i,0}$.
Then, iterate, for every $n\in\NN$,
 \begin{equation}\label{e:Algomain2}
 \begin{array}{|l}
(i)\;\operatorname{For}\;i=1,\ldots, m\\
\quad
\begin{array}{|l}
 c_{i,n} = x_{i,n} + \alpha_n(x_{i,n}- x_{i,n-1} )\\
\end{array}\\
(ii)\;\operatorname{For}\;k=1,\ldots, s\\
\quad
\begin{array}{|l}
 d_{k,n} = v_{k,n} + \alpha_n(v_{k,n}- v_{k,n-1} )\\
\end{array}\\
(iii)\;  \operatorname{For}\;i=1,\ldots, m\\
\quad
\begin{array}{|l}
 1.\; t_{i,n} =\sum_{k=1}^sL_{k,i}^*d_{k,n}+\rh_{i,n}\\
2.\; p_{i,n}:=\prox_{f_{i}}^{V_{i}^{-1}}\big(c_{i,n}-  V_{i}
(t_{i,n} -z_i)\big)\\
3.\; y_{i,n}:=2p_{i,n}-c_{i,n}\\
4.\; x_{i,n+1}:=x_{i,n}+\lambda_{n}(p_{i,n}-x_{i,n})\\
\end{array}\quad\\
(iv)\; \operatorname{For}\;k=1,\ldots, s\\
\quad
\begin{array}{|l}
1.\; u_{k,n} =\sum_{i=1}^m L_{k,i}y_{i,n}- \og_{k,n}\\
2.\; q_{k,n}:=\prox_{g_{k}^*}^{W_{k}^{-1}}
\big(d_{k,n}+W_{k}\big(u_{k,n}-r_k\big)\big)\\
3.\; v_{k,n+1}:=v_{k,n}+\lambda_{n}(q_{k,n}-v_{k,n}).\\
\end{array}
\end{array}
\end{equation} 
Set  
\begin{equation}
(\forall n\in\NN)\quad
\begin{cases}
 \xk_n:=(x_{1,n},\ldots, x_{m,n},v_{1,n},\ldots,v_{s,n})\\
\rk_n = (\rh_{1,n},\ldots, \rh_{m,n},\og_{1,n},\ldots,\og_{s,n})\\
\FF_n = \sigma(\xk_0,\ldots,\xk_n).\\
\end{cases}
\end{equation}
Suppose that the following conditions are satisfied.
\begin{enumerate}
\item\label{cond:onei} 
$ (\forall n\in \NN)\;
\E[\rk_{n}| \FF_n] = \big((\nabla_i\varphi(c_{1,n}\ldots, c_{m,n}))_{1\leq i\leq m},
 \nabla\ell_{1}^*(d_{1,n}),\ldots, \nabla\ell_{s}^*(d_{s,n})\big).
$
\item \label{cond:twoii}
$\sum_{n\in\NN}
\E[ \| \sum_{i=1}^m \rh_{i,n}-\nabla_i\varphi(c_{1,n}\ldots, c_{m,n}) \|^2
+ \sum_{k=1}^s\|  \og_{k,n}-\nabla\ell_{k}^*(d_{k,n}) \|^2   
| \FF_n] < +\infty.$
\item\label{cond:twoiii}
 $\max_{1\leq i\leq m}\sup_{n\in\NN}\| x_{i,n}-x_{i,n-1}\| < \infty $ a.s.
and  $\max_{1\leq k\leq s}\sup_{n\in\NN}\| v_{k,n}-v_{k,n-1}\|< \infty $ a.s.,
and $\sum_{n\in\NN}\alpha_n < +\infty $.
\end{enumerate}
Then the following hold for some random vector 
$(\overline{x}_1,\ldots,\overline{x}_m, \overline{v}_1,\ldots,\overline{v}_s)$,
$\mathcal{P}\times\mathcal{D}$-valued almost surely.
\begin{enumerate}
 \item
\label{t:2i} $(\forall i\in\{1,\ldots,m\})$\; $x_{i,n}\weakly \overline{x}_i$ and 
$(\forall k\in\{1,\ldots,s\})$\; $v_{k,n}\weakly\overline{v}_k$ almost surely.
\item \label{t:2ii}
Suppose that the function $\varphi$ 
is uniformly convex at 
$(\overline{x}_1,\ldots,\overline{x}_m)$, then $(\forall i\in\{1,\ldots,m\})\; $
$x_{i,n}\to \overline{x}_i$ almost surely. 
\item 
\label{t:2iii}
Suppose that there exists $j\in\{1,\ldots,m\}$ such that $\ell_{j}^{*}$ is uniformly convex
 at $\overline{v}_j$, then $v_{j,n}$
$\to \overline{v}_j$ almost surely. 
\item
\label{t:2iv}
Suppose that  $(\forall  (x_1,\ldots,x_m)\in \KK_1\times\ldots\times\KK_m)\; \varphi(x_1,\ldots,x_m) = \sum_{i=1}^mh_i(x_i)$ where
each $h_i\in\Gamma_0(\KK_i)$ is a convex differentiable function,   
there exists $j\in\{1,\ldots,m\}$ such that $h_j$ is uniformly convex at $\overline{x}_j$, then
$x_{j,n}\to \overline{x}_j$ almost surely. 
\end{enumerate}
\end{corollary}
\begin{proof}
Using the same argument as in the proof of \cite[Corollary 5.1]{jota1}, Example \ref{ex:prob2a} reduces 
to special case of Problem \ref{probbb} with 
\begin{equation}
\label{e:sett1}
 \begin{cases}
  (\forall i\in\{1,\ldots,m\})\quad A_i =\partial f_i\quad
 \text{and}\quad C_i =\nabla_i\varphi,\\
 (\forall k\in\{1,\ldots,s\})\quad B_k =\partial g_k\quad \text{and}\quad D_k =\partial\ell_k.\\
 \end{cases}
\end{equation}
Furthermore, by \eqref{set0} and \eqref{e:equi}, the algorithm \eqref{e:Algomain2} is a special case of 
the algorithm \eqref{e:Algomain}. Therefore, the conclusions follow from Theorem \ref{t:2}.
\end{proof}
\begin{remark} Here are some remarks.
\begin{enumerate}
\item
The algorithm proposed in this section is new, also if  $(\forall n\in\NN)\;\alpha_n=0$, 
since a stochastic algorithm for system of monotone inclusions involving both non-smooth 
coupling and smooth coupling is not available in the literature.  The stochastic algorithms  for either solving smooth coupling 
or non-smooth coupling can be found in \cite[Section 5.2]{plc14} or   \cite[Section 4]{pesquet14}.
In the case when $(\forall n\in\NN)\; \alpha_n=0$, we obtain the stochastic extension
of the framework in \cite[Section 4 and 5]{jota1} and in \cite[Section 4.2]{icip14}. 
Furthermore, in the special case when $m=s=1$, we obtain
a stochastic version of the inertial primal-dual algorithm in \cite{Dirk}. See also \cite{Cham14}
for related results.
\item 
Sufficient conditions, which ensure that the condition \eqref{e:bfi} is satisfied, are provided in 
\cite[Proposition 5.3]{30Combettes13}. For instance, \eqref{e:bfi} holds if \eqref{primal2} has at least one solution and 
$(r_1,\ldots, r_s)$ belongs to the strong relative interior of the following set
\[
 \Menge{\Big(\sum_{i=1}^m L_{k,i}x_i-v_k\Big)_{1\leq k\leq s}}{
\begin{cases}
 (\forall i\in\{1,\ldots,m\})\; x_i\in\dom f_i\\
 (\forall k\in\{1,\ldots,s\})\; v_k\in\dom g_k+\dom\ell_k
\end{cases}
}.
\]
\end{enumerate} 
\end{remark}
\subsection{A second class of the stochastic inertial primal-dual splitting methods}
In this subsection, 
we will derive a new class of stochastic inertial primal-dual splitting methods, for the case 
where $(\forall i\in\{1,\ldots,m\})\; A_i =0$. This class of algorithms corresponds 
to the choice $U_1 = \TT$ in Lemma \ref{l:2a}.
\begin{theorem}
\label{mainal1a} 
In Problem \ref{probbb}, set $(\forall i \in \{1,\ldots,m\})\; A_i =0$. 
Let $\beta$ be defined as in \eqref{beta}, and assume $2\beta > 1$.
Let $\varepsilon \in \left]0,\min\{1,\beta\}\right[$, and 
let $(\lambda_n)_{n\in\NN}$ be a sequence 
in $\left[\varepsilon,1\right]$,
let $(\alpha_n)_{n\in\NN}$ be a sequence in $\left[0,1-\varepsilon\right]$.
For every $i\in \{1,\ldots, m\}$, 
 let $(\rh_{i,n})_{n\in\NN}$ be a $\KK_i$-valued, squared integrable random process, and let 
$x_{i,0}$  be a $\KK_i$-valued, squared integrable random vector and set $x_{i,-1}=x_{i,0}$.  
For every $k\in \{1,\ldots,s\}$, 
let $W_{k}\in\BL(\GG_k)$ be self-adjoint and strongly positive,
let $(\og_{i,n})_{n\in\NN}$ be a $\GG_i$-valued, squared integrable random process, and let
$v_{i,0}$  be a $\GG_i$-valued, squared integrable random vector and set $v_{i,-1}=v_{i,0}$.
Then, iterate, for every $n\in\NN$,
\begin{equation}
\label{e:Algomain3b}
 \begin{array}{|l}
(i)\;\operatorname{For}\;i=1,\ldots, m\\
\quad
\begin{array}{|l}
 c_{i,n} = x_{i,n} + \alpha_n(x_{i,n}- x_{i,n-1} )\\
\end{array}\\
(ii)\;\operatorname{For}\;k=1,\ldots, s\\
\quad
\begin{array}{|l}
 d_{k,n} = v_{k,n} + \alpha_n(v_{k,n}- v_{k,n-1} )\\
\end{array}\\
(iii)\;  \operatorname{For}\;i=1,\ldots, m\\
\quad
\begin{array}{|l}
 1.\; s_{i,n} =  c_{i,n}-V_i(\rh_{i,n}-z_i)\\
2.\; y_{i,n} = s_{i,n} -  V_i\sum_{k=1}^s L_{k,i}^*d_{k,n}
\end{array}\quad\\
(iv)\; \operatorname{For}\;k=1,\ldots, s\\
\quad
\begin{array}{|l}
1.\; q_{k,n}= J_{W_kB_{k}^{-1}}
\big(d_{k,n}+W_{k}\big(\sum_{i=1}^m L_{k,i}y_{i,n} - \og_{k,n}-r_k\big)\big)\\
2.\; v_{k,n+1}=v_{k,n}+\lambda_{n}(q_{k,n}-v_{k,n}).\\
\end{array}\\
(v)\;  \operatorname{For}\;i=1,\ldots, m\\
\quad
\begin{array}{|l}
 1.\; p_{i,n} =  s_{i,n}-V_i\sum_{k=1}^sL_{k,i}^*q_{k,n}\\
2.\; x_{i,n+1}=x_{i,n}+\lambda_{n}(p_{i,n}-x_{i,n}).\\
\end{array}\quad\\
\end{array}
\end{equation} 
Set  
\begin{equation}
\label{e:so}
(\forall n\in\NN)\quad
\begin{cases}
 \xk_n:=(x_{1,n},\ldots, x_{m,n},v_{1,n},\ldots,v_{s,n})\\
\rk_n = (\rh_{1,n},\ldots, \rh_{m,n},\og_{1,n},\ldots,\og_{s,n})\\
\FF_n = \sigma(\xk_0,\ldots,\xk_n).\\
\end{cases}
\end{equation}
Suppose that the following conditions are satisfied.
\begin{enumerate}
\item\label{cond:onei} 
$ (\forall n\in \NN)\;
\E[\rk_{n}| \FF_n] = \big((C_i(c_{1,n}\ldots, c_{m,n}))_{1\leq i\leq m}, D_{1}^{-1}d_{1,n},\ldots, D_{s}^{-1}d_{s,n}\big).
$
\item \label{cond:twoii}
$\sum_{n\in\NN}
\E[ \| \sum_{i=1}^m \rh_{i,n}-C_i(c_{1,n}\ldots, c_{m,n})\|^2
+ \sum_{k=1}^s\|  \og_{k,n}- D_{k}^{-1}d_{k,n}\|^2   
| \FF_n] < +\infty.$
\item\label{cond:twoiii}
 $\max_{1\leq i\leq m}\sup_{n\in\NN}\| x_{i,n}-x_{i,n-1}\|<\infty$ a.s.
and  $\max_{1\leq k\leq s}\sup_{n\in\NN}\| v_{k,n}-v_{k,n-1}\|<\infty$ a.s.,
and $\sum_{n\in\NN}\alpha_n < +\infty $.
\end{enumerate}
Then the following hold for some random vector 
$(\overline{x}_1,\ldots,\overline{x}_m, \overline{v}_1,\ldots,\overline{v}_s)$,
$\mathcal{P}\times\mathcal{D}$-valued almost surely.
\begin{enumerate}
 \item
\label{t:2i} $(\forall i\in\{1,\ldots,m\})$\; $x_{i,n}\weakly \overline{x}_i$ and 
$(\forall k\in\{1,\ldots,s\})$\; $v_{k,n}\weakly\overline{v}_k$ almost surely.
\item \label{t:2ii}
Suppose that the operator $(x_i)_{1\leq i\leq m} \mapsto (C_j(x_i)_{1\leq i\leq m})_{1\leq j\leq m}$ 
is demiregular at 
$(\overline{x}_1,\ldots,\overline{x}_m)$, then $(\forall i\in\{1,\ldots,m\})\; $
$x_{i,n}\to \overline{x}_i$ almost surely. 
\item 
\label{t:2iii}
Suppose that there exists $j\in\{1,\ldots,m\}$ such that $D_{j}^{-1}$ is demiregular 
at 
$\overline{v}_j$, then $ 
v_{j,n}$
$\to \overline{v}_j$ almost surely. 
\item
\label{t:2iv}
Suppose that there exists $j\in\{1,\ldots,m\}$ and an operator 
$C\colon\KK_j \to \KK_j$ such that 
$(\forall(x_i)_{1\leq i\leq m}\in(\KK_i)_{1\leq i\leq m})\; C_j(x_1,\ldots,x_m) = Cx_j$ and 
$C$ is demiregular at $\overline{x}_j$, then
$x_{j,n}\to \overline{x}_j$ almost surely. 
\end{enumerate}
\end{theorem}
\begin{proof} 
Set 
\begin{equation}
\label{e:so1}
\begin{cases}
\xx_{n} &= (x_{1,n},\ldots,x_{m,n})\\
\cc_{n} &= (c_{1,n},\ldots,c_{m,n})\\
\sss_{n} &= (s_{1,n},\ldots,s_{m,n})\\
\yy_{n} &= (y_{1,n},\ldots,y_{m,n})\\
\pp_{n} &= (p_{1,n},\ldots,p_{m,n})\\
\end{cases}
\quad
\text{and}
\quad
\begin{cases}
\vv_{n} &= (v_{1,n},\ldots,v_{s,n})\\
\dd_{n} &= (d_{1,n},\ldots,d_{s,n})\\
\qq_{n} &= (q_{1,n},\ldots,q_{s,n})\\
\bb_{n} &= (\og_{1,n},\ldots,\og_{s,n})\\
\aaa_{n} &= (\rh_{1,n},\ldots,\rh_{m,n}).
\end{cases}
\end{equation}
Using the notation introduced in \eqref{e:so} and \eqref{e:so1}, the definition of the operators $\VV,\WW, \LL$ in \eqref{VeW}, 
the definition of $\BB$ and 
the reference vectors $\zz,\rr$ as in \eqref{e:acl1},
we can rewrite \eqref{e:Algomain3b} as 
\begin{equation}
\label{e:Algomain3c}
(\forall n\in\NN)\quad
 \begin{array}{|l}
 \cc_{n} = \xx_{n} + \alpha_n(\xx_{n}- \xx_{n-1} )\\
 \dd_{n} = \vv_{n} + \alpha_n(\vv_{n}- \vv_{n-1} )\\
  \sss_{n} =  \cc_{n}-\VV(\aaa_{n}-\zz)\\
 \yy_{n} = \sss_{n} -  \VV \LL^*\dd_{n}\\
 \qq_{n}= J_{\WW \BB^{-1}}
\big(\dd_{n}+\WW\big(\LL \yy_{n} - \bb_{n}-\rr\big)\big)\\
 \pp_{n} =  \sss_{n}-\VV \LL^*\qq_{n}\\
 \xx_{n+1}=\xx_{n}+\lambda_{n}(\pp_{n}-\xx_{n})\\
 \vv_{n+1}=\vv_{n}+\lambda_{n}(\qq_{n}-\vv_{n}).\\
\end{array}
\end{equation} 
Now, we have
\begin{alignat}{2}
\label{eee1}
&\quad(\forall n\in\NN)\quad\qq_n= J_{\WW \BB^{-1}}
\big(\dd_{n}+\WW\big(\LL \yy_{n} - \bb_{n}-\rr\big)\big)\notag\\
&\Leftrightarrow (\forall n\in\NN)\quad \WW^{-1}(\dd_n-\qq_n) + (\LL\yy_n-\bb_n-\rr)\in \BB^{-1}\qq_n\notag\\
&\Leftrightarrow (\forall n\in\NN)\quad \WW^{-1}(\dd_n-\qq_n) + (\LL\sss_n-\LL\VV\LL^*\dd_n)-\bb_n\in\rr + \BB^{-1}\qq_n\notag\\
&\Leftrightarrow  (\forall n\in\NN)\quad \WW^{-1}(\dd_n-\qq_n) -\LL\VV\LL^*(\dd_n-\qq_n)+ \LL\sss_n -\LL\VV\LL^*\qq_n-\bb_n
\in\rr + \BB^{-1}\qq_n\notag\\
&\Leftrightarrow (\forall n\in\NN)\quad \WW^{-1}(\dd_n-\qq_n) -\LL\VV\LL^*(\dd_n-\qq_n)+ \LL\pp_n-\bb_n
\in\rr + \BB^{-1}\qq_n.
\end{alignat} 
Since $\AAA =0$, we next have 
\begin{alignat}{2}
\label{eee2}
&\quad(\forall n\in\NN)\quad \pp_n = \sss_n-\VV\LL^*\qq_n\notag\\
&\Leftrightarrow\quad(\forall n\in\NN)\quad \VV^{-1}(\cc_n-\pp_n)-\LL^*\qq_n-\aaa_n \in -z + \AAA\pp_n.
\end{alignat}
Now, by setting $(\forall n\in\NN)\; \yk_n = (\pp_n,\qq_n), \ck_n = (\cc_n,\dd_n)$, 
by using \eqref{ee2} and the definition of $\TT$ in \eqref{e:UeT} and $\MM, \SSS$ in \eqref{e:2a}, we obtain
\begin{equation}
(\forall n\in\NN)\quad\TT^{-1}(\ck_n-\yk_n) -\rk_n \in \MM\yk_n + \SSS\yk_n,
\end{equation}
which is equivalent to 
\begin{equation}
(\forall n\in\NN)\quad \yk_n = J_{\TT(\MM+\SSS)}( \ck_n- \TT\rk_n) 
= J_{\TT(\MM+\SSS)}( \ck_n- \TT\rk_n).
\end{equation}
Therefore, \eqref{e:Algomain3c} becomes 
\begin{equation}
\label{e:Algomain3d}
(\forall n\in\NN)\quad
 \begin{array}{|l}
 \ck_n = \xk_{n} + \alpha_n(\xk_{n}- \xk_{n-1} )\\
  \yk_n = J_{\TT(\MM+\SSS)}( \ck_n- \TT\rk_n)\\
\xk_{n+1}=\xk_{n}+\lambda_{n}(\yk_n-\xk_{n}),\\
 \end{array}
\end{equation} 
which is a special instance of the iteration \eqref{e:main1*} with 
$(\forall n\in\NN)\;\gamma_n =1\in \left]\varepsilon,(2-\varepsilon)\beta \right] $.
We next see that conditions \ref{cond:onei}-\ref{cond:twoiii} can be rewritten in the space $\HHH$ 
defined in \eqref{e:HeK} as 
\begin{enumerate}
\item[(a1)]\label{cs} 
$ (\forall n\in \NN)\;
\E[\rk_{n}| \FF_n] = \QQ\ck_n.
$
\item [(b1)] \label{css}
$\sum_{n\in\NN} \E[ \|  \rk_{n}-\QQ\ck_n\|^2| \FF_n] < +\infty$ a.s.
\item[(c1)] $\sup_{n\in\NN}\| \xk_{n}-\xk_{n-1}\| < \infty$ a.s.
and $\sum_{n\in\NN}\alpha_n < +\infty $.
\end{enumerate}
Therefore, all the assumptions in Algorithm \ref{a:al1} and Theorem \ref{t:1} are satisfied.

\ref{t:2i}: In view of Theorem \ref{t:1}\ref{t:1i}, $\xk_n\weakly (\overline{\xx},\overline{\vv})$ which is equivalent to 
$(\forall i\in\{1,\ldots,m\})\; x_{i,n}\weakly \overline{x}_i$ and $(\forall k\in\{1,\ldots,s\})\; v_{k,n}\weakly \overline{v}_k$.

\ref{t:2ii}$\&$ \ref{t:2iii}: By Theorem  \ref{t:1}\ref{t:1ia}, we have $\QQ\xk_n\to \QQ\overline{\xk}$ which is equivalent to 
\begin{equation}
\begin{cases}
 (\forall i\in\{1,\ldots,m\})\quad C_i(x_{1,n},\ldots,x_{m,n}) \to C_i(\overline{x}_{1},\ldots,\overline{x}_m)\\
 (\forall k\in\{1,\ldots,s\})\quad D_{k}^{-1}v_{k,n} \to D_{k}^{-1}\overline{v}_k.\\
\end{cases}
\end{equation}
Therefore, if the operator $(x_i)_{1\leq i\leq m} \mapsto (C_j(x_i)_{1\leq i\leq m})_{1\leq j\leq m}$ is demiregular at 
$(\overline{x}_1,\ldots,\overline{x}_m)$, we obtain $(\forall i\in\{1,\ldots,m\})\; x_{i,n}\to\overline{x}_i $. By the same 
reason,  if there exists $j\in \{1,\ldots,m\}$ such that $D_{j}^{-1}$ is demiregular at $\overline{v}_j$, the 
$v_{j,n}\to\overline{v}_j $.

\ref{t:2iv} This conclusion follows from the definition of the demiregular operators as in \ref{t:2iii}.
\end{proof}
\begin{corollary}
\label{mainal1a} 
In Example \ref{ex:prob2a}, set $(\forall i \in \{1,\ldots,m\})\; f_i =0$.
For every $i\in \{1,\ldots, m\}$, 
let $V_{i}\in\BL(\KK_i)$ be self-adjoint and strongly positive.
 Let $\nu_0$  be a strictly positive number such that \eqref{cos1} is satisfied.
For every $k\in \{1,\ldots, s\}$, 
let $W_{k}\in\BL(\GG_k)$ be self-adjoint and strongly positive.
 Let $\mu_0$  be a strictly positive number such that \eqref{cos2} is satisfied. 
Let $\beta$ be defined as in \eqref{beta} such that $2\beta > 1$, 
 let $\varepsilon \in \left]0,\min\{1,\beta\}\right[$, and 
let $(\lambda_n)_{n\in\NN}$ be a sequence 
in $\left[\varepsilon,1\right]$, 
let $(\alpha_n)_{n\in\NN}$ be a sequence in $\left[0,1-\varepsilon\right]$.
For every $i\in \{1,\ldots, m\}$, 
 let $(\rh_{i,n})_{n\in\NN}$ be a $\KK_i$-valued, squared integrable random process, and let
$x_{i,0}$  be a $\KK_i$-valued, squared integrable random vector and set $x_{i,-1}=x_{i,0}$.  
For every $k\in \{1,\ldots,s\}$, 
let $W_{k}\in\BL(\GG_k)$ be  self-adjoint and strongly positive,
let $(\og_{i,n})_{n\in\NN}$ be a $\GG_i$-valued, squared integrable random process, and let
$v_{i,0}$  be a $\GG_i$-valued, squared integrable random vector and set $v_{i,-1}=v_{i,0}$.
Then, iterate, for every $n\in\NN$,
\begin{equation}
\label{e:Algomain3b}
 \begin{array}{|l}
(i)\;\operatorname{For}\;i=1,\ldots, m\\
\quad
\begin{array}{|l}
 c_{i,n} = x_{i,n} + \alpha_n(x_{i,n}- x_{i,n-1} )\\
\end{array}\\
(ii)\;\operatorname{For}\;k=1,\ldots, s\\
\quad
\begin{array}{|l}
 d_{k,n} = v_{k,n} + \alpha_n(v_{k,n}- v_{k,n-1} )\\
\end{array}\\
(iii)\;  \operatorname{For}\;i=1,\ldots, m\\
\quad
\begin{array}{|l}
 1.\; s_{i,n} =  c_{i,n}-V_i(\rh_{i,n}-z_i)\\
2.\; y_{i,n} = s_{i,n} -  V_i\sum_{k=1}^s L_{k,i}^*d_{k,n}
\end{array}\quad\\
(iv)\; \operatorname{For}\;k=1,\ldots, s\\
\quad
\begin{array}{|l}
1.\; q_{k,n}= \prox_{g_{k}^{*}}^{W_{k}^{-1}}
\big(d_{k,n}+W_{k}\big(\sum_{i=1}^m L_{k,i}y_{i,n} - \og_{k,n}-r_k\big)\big)\\
2.\; v_{k,n+1}=v_{k,n}+\lambda_{n}(q_{k,n}-v_{k,n}).\\
\end{array}\\
(v)\;  \operatorname{For}\;i=1,\ldots, m\\
\quad
\begin{array}{|l}
 1.\; p_{i,n} =  s_{i,n}-V_i\sum_{k=1}^sL_{k,i}^*q_{k,n}\\
2.\; x_{i,n+1}=x_{i,n}+\lambda_{n}(p_{i,n}-x_{i,n}).\\
\end{array}\quad\\
\end{array}
\end{equation} 
Set  
\begin{equation}
\label{e:so2}
(\forall n\in\NN)\quad
\begin{cases}
 \xk_n=(x_{1,n},\ldots, x_{m,n},v_{1,n},\ldots,v_{s,n})\\
\rk_n = (\rh_{1,n},\ldots, \rh_{m,n},\og_{1,n},\ldots,\og_{s,n})\\
\FF_n = \sigma(\xk_0,\ldots,\xk_n).\\
\end{cases}
\end{equation}
Suppose that the following conditions are satisfied.
\begin{enumerate}
\item\label{cond:onei} 
$(\forall n\in \NN)\;
\E[\rk_{n}| \FF_n] = \big((\nabla_i\varphi(c_{1,n}\ldots, c_{m,n}))_{1\leq i\leq m},
 \nabla\ell_{1}^*(d_{1,n}),\ldots, \nabla\ell_{s}^*(d_{s,n})\big).
$
\item \label{cond:twoii}
$\sum_{n\in\NN}
\E[ \| \sum_{i=1}^m \rh_{i,n}-\nabla_i\varphi(c_{1,n}\ldots, c_{m,n}) \|^2
+ \sum_{k=1}^s\|  \og_{k,n}-\nabla\ell_{k}^*(d_{k,n}) \|^2   
| \FF_n] < +\infty.$
\item\label{cond:twoiii}
 $\max_{1\leq i\leq m}\sup_{n\in\NN}\| x_{i,n}-x_{i,n-1}\| < \infty $ a.s.
and  $\max_{1\leq k\leq s}\sup_{n\in\NN}\| v_{k,n}-v_{k,n-1}\|< \infty $ a.s.,
and $\sum_{n\in\NN}\alpha_n < +\infty $.
\end{enumerate}
Then the following hold for some random vector 
$(\overline{x}_1,\ldots,\overline{x}_m, \overline{v}_1,\ldots,\overline{v}_s)$,
$\mathcal{P}\times\mathcal{D}$-valued almost surely.
\begin{enumerate}
 \item
\label{t:2i} $(\forall i\in\{1,\ldots,m\})$\; $x_{i,n}\weakly \overline{x}_i$ and 
$(\forall k\in\{1,\ldots,s\})$\; $v_{k,n}\weakly\overline{v}_k$ almost surely.
\item \label{t:2ii}
Suppose that the function $\varphi$ 
is uniformly convex at 
$(\overline{x}_1,\ldots,\overline{x}_m)$, then $(\forall i\in\{1,\ldots,m\})\; $
$x_{i,n}\to \overline{x}_i$ almost surely. 
\item 
\label{t:2iii}
Suppose that there exists $j\in\{1,\ldots,m\}$ such that $\ell_{j}^{*}$ is uniformly convex
 at $\overline{v}_j$, then $v_{j,n}$
$\to \overline{v}_j$ almost surely. 
\item
\label{t:2iv}
Suppose that  $(\forall  (x_1,\ldots,x_m)\in \KK_1\times\ldots\times\KK_m)\; \varphi(x_1,\ldots,x_m) = \sum_{i=1}^mh_i(x_i)$ where
each $h_i\in\Gamma_0(\KK_i)$ is a convex differentiable function,   
there exists $j\in\{1,\ldots,m\}$ such that $h_j$ is uniformly convex at $\overline{x}_j$, then
$x_{j,n}\to \overline{x}_j$ almost surely. 
\end{enumerate}
\end{corollary}

\begin{remark}
Let $\tau\in\left]0,+\infty\right[$ and $\sigma\in\left]0,+\infty\right[$. 
If in Problem~\ref{probbb},  $(\forall i\in\{1,\ldots,m\})\; V_i=\tau\Id$ and $C_i$ is $\nu$ cocoercive
for some $\nu\in\left]0,+\infty\right[$,  and 
$(\forall k\in\{1,\ldots,s\})\; W_k=\sigma\Id$ and $D_k^{-1}$ is $\mu$ cocoercive for some $\mu\in\left]0,+\infty\right[$,
then the condition $2\beta>1$ in Algorithm~\ref{mainal1a} is satisfied for $\tau$ and $\sigma$ sufficiently small. 
Indeed, in this case $\beta=\min\{\nu/\tau, (\mu/\sigma)(1-\tau\sigma\|\LL\|^2)\}$.
\end{remark}

\begin{remark} 
The  operator $\TT$ has been first considered in \cite{icip14}, 
and then used in \cite{Davis14,pesquet14}, in the deterministic setting. The
results of this subsection constitute an extension of \cite[Section 4.2]{icip14}
to the stochastic and inertial setting.  See \cite{icip14} for the connections to \cite{Chen11} 
and \cite{loris11}.
\end{remark}

\noindent{\small{\bf Acknowldgments} 
This material is based upon work supported by the Center for Brains, Minds and Machines (CBMM), funded by NSF STC award CCF-1231216.
L. Rosasco acknowledges the financial support of the Italian Ministry of Education, University and Research FIRB project RBFR12M3AC.
S. Villa is member of the Gruppo Nazionale per
l'Analisi Matematica, la Probabilit\`a e le loro Applicazioni (GNAMPA)
of the Istituto Nazionale di Alta Matematica (INdAM). 
Bang Cong Vu's research work is partially funded by Vietnam National Foundation for Science
and Technology Development (NAFOSTED) under Grant No. 102.01-2014.02.  }


\begin{thebibliography}{99}
\bibitem{Attouch01}
F. Alvarez and H. Attouch,
 An inertial proximal method for maximal monotone operators
via discretization of a nonlinear oscillator with damping,
{\em Set-Valued Analysis}, vol. 9, pp. 3--11,
2001.

\bibitem{plc2010}
H. Attouch, L. M. Brice\~no-Arias, and   P. L. Combettes, 
A parallel splitting method for coupled monotone
inclusions,
{\em SIAM J. Control Optim.}, 
{vol. 48}, 
pp. 3246--3270, 2010.

\bibitem{livre1}
H. H. Bauschke and P. L. Combettes, 
{\em Convex Analysis and Monotone Operator Theory in Hilbert Spaces}.
Springer, New York, 2011.

\bibitem{Barty07}
K.~Barty ,  J.-S.~Roy and C.~Strugarek,
Hilbert-valued perturbed subgradient algorithms,
{\em Math. Oper. Res.},
vol. 32, pp. 551--562, 2007.

\bibitem{Bennar07}
A.~Bennar and J.-M.~Monnez,
Almost sure convergence of a stochastic approximation process in a convex set,
{\em Int. J. Appl. Math.}, vol. 20, pp. 713-722, 2007.

\bibitem{BiaHacIut14}
P. Bianchi, W. Hachem and F. Iutzeler,
A Stochastic Coordinate Descent Primal-Dual Algorithm and Applications to Large-Scale Composite Optimization,
http://arxiv.org/abs/1407.0898.

\bibitem{Bric13} 
L. M. Brice\~{n}o-Arias and P. L. Combettes, 
Monotone operator methods for Nash equilibria in non-potential 
games, in
\emph{Computational and Analytical Mathematics,} 
(D. Bailey, H. H. Bauschke, P. Borwein, F. Garvan, M. Th\'era, 
J. Vanderwerff, and H. Wolkowicz, eds.).
Springer, New York, 2013.

\bibitem{Cham14}
A. Chambolle and T. Pock, 
On the ergodic convergence rates of a first-order primal-dual algorithm, preprint, 2014.
http://www.optimization-online.org/DBFILE/2014/09/4532.pdf

\bibitem{Chen11}
P. Chen, J. Huang, and X. Zhang, A primal--dual fixed point
algorithm for convex separable minimization with applications
to image restoration,
{\em Inverse Problems}, vol. 29, no. 2, 2013,
doi:10.1088/0266-5611/29/2/025011.

\bibitem{CheRoc97}
G. H.-G. Chen and T. Rockafellar, Convergence rates in forward-backward splitting,
{\em SIAM J. Optim.}, vol. 7, no. 2, 1997.


\bibitem{30Combettes13}
P. L. Combettes,  
Systems of structured monotone inclusions: duality,
algorithms, and applications,
{\em SIAM J. Optim.}, vol. 23, pp. 2420-2447, 2013.


\bibitem{plc04}
P. L. Combettes,  
Solving monotone inclusions via compositions of nonexpansive averaged operators,
{\em Optimization},
vol. 53, pp. 475--504, 2004.

\bibitem{icip14}
P. L Combettes, L. Condat, J.-C. Pesquet, and B. C. V\~u,
A forward-backward view of some primal-dual
optimization methods in image recovery, 
{\em In Proc. Int. Conf. Image Process.}, Paris, France, 27-30 Oct.
2014

\bibitem{optim2}
P. L. Combettes and B. C. V\~u, 
Variable metric forward-backward splitting 
with applications to monotone inclusions in 
duality, { \em Optimization}, vol. 63, pp. 1289-1318, 2013.

\bibitem{plc14}
P. L. Combettes and J.-C. Pesquet,
Stochastic quasi-Fej\'er block-coordinate fixed point iterations with random sweeping, preprint, 
{\em SIAM J. Optim.}, to appear. 

\bibitem{siam05} 
P. L. Combettes and  V. R.  Wajs,  
Signal recovery by proximal forward-backward splitting,
{\em Multiscale Model.  Simul.} 
{vol. 4},
 pp. 1168--1200, 2005. 
 
\bibitem{Con13}
L. Condat,
A primal-dual Splitting Method for Convex Optimization Involving Lipschitzian, Proximable and Linear Composite Terms, 
{\em J. Optim. Theory Appl.} 
vol. 158, 
pp. 460--479, 
2013.

\bibitem{Dau04}
I. Daubechies, M. Defrise, and C. De Mol,
An Iterative Thresholding Algorithm for Linear Inverse Problems with a Sparsity Constraint,
{\em Comm. Pure Appl. Math.}
vol. 57, pp. 1413--1457, 2004.

\bibitem{Davis14}
 D. Davis,
Convergence rate analysis of primal-dual splitting schemes, preprint, 2014.
http://arxiv.org/abs/1408.4419

\bibitem{Devi11}
E. De Vito, V. Umanit\`a, and S. Villa,
A consistent algorithm to solve Lasso, elastic-net and
Tikhonov regularization,
{\em J. Complexity,}
vol. 27, pp. 188--200, 2011.

\bibitem{Duch09}
J. Duchi and Y. Singer, 
Efficient online and batch learning using forward backward 
splitting,
{\em J. Mach. Learn. Res.,}
vol. 10, pp. 2899--2934, 2009.

\bibitem{Duret04}
R.~ Durrett,
	{\em Probability: theory and example.}
	{Cambridge Series in Statistical and Probabilistic Mathematics}.
	Cambridge University Press, Cambridge,
	2010.

\bibitem{Facc03} 
F. Facchinei and J.-S. Pang,
{\em Finite-Dimensional Variational Inequalities and 
Complementarity Problems.}
Springer-Verlag, New York, 2003.

\bibitem{For95}
R. Fortet, {\em Vecteurs, fonctions et distributions al\'eatoires dans les
              espaces de {H}ilbert},
{Editions Herm\`es, Paris}, 1995.
	
\bibitem{Glow89}
R. Glowinski and P. Le Tallec,
{\em Augmented Lagrangian and Operator-Splitting
Methods in Nonlinear Mechanics.}
SIAM, Philadelphia, 1989.

\bibitem{Har81}
A. Haraux, 
{\em Nonlinear Evolution Equations: Global Behavior of Solutions}, 
Lecture Notes in Math., vol. 841, Springer-Verlag, New York, 1981.

\bibitem{LedTal91}
M. Ledoux and M. Talagrand,
{\em Probability in Banach Spaces: Isoperimetry and Processes}.
Springer, New York 1991.

\bibitem{mercier79}
B. Mercier,   
{\em Topics in Finite Element Solution of 
Elliptic Problems },
Lectures on Mathematics, no. 63.
Tata Institute of Fundamental Research, Bombay (1979)

\bibitem{Merc80} 
B. Mercier,  
{\em In\'equations Variationnelles de la M\'ecanique}
(Publications Math\'ematiques d'Orsay, no. 80.01).
Universit\'e de Paris-XI, Orsay, France, 1980. 

\bibitem{Mor62}
J. J. Moreau,
{Fonctions convexes duales et points proximaux dans un espace
hilbertien.}
{\em C. R. Acad. Sci. Paris S\'er. A}, 
vol. 255, pp. 2897--2899, 1962.

\bibitem{MRSVV10}
S.~Mosci, L.~Rosasco, M.~Santoro, A.~Verri and S.~Villa,
Solving Structured Sparsity Regularization with Proximal
               Methods,  in           
{\em Machine Learning and Knowledge discovery in Databases European Conference},
pp. 418-433, 2010.

\bibitem{Moudafi03}
A. Moudafi and M. Oliny,
 Convergence of a splitting inertial proximal
method for monotone operators,
{\em J. of Computational
and Applied Mathematics}, vol. 155, pp. 447-454, 2003.

\bibitem{Nem09}
A. Nemirovski, A. Juditsky, G. Lan, and A. Shapiro,
Robust stochastic approximation approach to stochastic programming,
{\em SIAM J. Optim.}, vol. 19, pp. 1574--1609, 2008.

\bibitem{Nes83}
Y. Nesterov, 
A method for unconstrained convex minimization problem with the rate of convergence $O(1/k^2)$,
{\em Doklady AN SSSR}, vol. 269, pp. 543--547, 1983.

\bibitem{JC10}
J. C.  Pesquet and N. Pustelnik,  
A parallel inertial proximal optimization method,
{\em Pacific Journal of Optimization}
{vol.  8},
pp. 273-306, 2012. 

\bibitem{pesquet14}
J.-C. Pesquet and A. Repetti,
A class of randomized primal-dual algorithms for distributed optimization,
 preprint 
{arXiv:1406.6404}, 2014.

\bibitem{Polyak64}
B. T. Polyak,
 Some methods of speeding up the convergence of
iteration methods,
{\em U.S.S.R. Comput. Math. Math. Phys.}, vol. 4, pp. 1--17, 1964.

\bibitem{Dirk}
D. A. Lorenz and T. Pock,
An inertial forward-backward method for monotone inclusions
{\em J. Math. Imaging Vis.}, vol.51, pp.311--325, 
 2015.

\bibitem{loris11}
I. Loris and C. Verhoeven, On a generalization of the iterative
soft-thresholding algorithm for the case of non-separable
penalty,
{\em Inverse Problems}, vol. 27, no. 12, p. 125007, 2011.

\bibitem{Raguet11}
H.Raguet, J. Fadili, and  G. Peyr\'e, 
Generalized forward-backward
splitting,
{\em SIAM J. Imaging Sci.},
{vol. 6}, 
pp. 1199--1226, 2013.

\bibitem{Xu08}
H. Jiang and H. Xu, 
Stochastic approximation approaches to the stochastic variational inequality problem,
{\em IEEE Trans. Automat. Control}, vol. 53, pp. 1462--1475, 2008.

\bibitem{RosVilMos13}
L. Rosasco, S. Villa, S. Mosci, M.
              Santoro, M. and  A. Verri,
Nonparametric sparsity and regularization,
{\em J. Mach. Learn. Res.}, vol. 14, pp. 1665--1714, 2013.


\bibitem{LSB14b}
L. Rosasco, S. Villa, and B. C.  V\~u,  
A Stochastic forward-backward splitting method for solving monotone inclusions in 
Hilbert spaces, arXiv:1403.7999, 2014.

\bibitem{LSB14a}
L. Rosasco, S. Villa, and B. C.  V\~u,  
Convergence of stochastic proximal gradient,
  arXiv:1403.5074, 2014.

\bibitem{Rob85} 
H. Robbins and
D. Siegmund, 
A Convergence theorem for non negative almost supermartingales and some applications,
in: {\em Optimizing Methods in Statistic}, (J. S. Rustagi, Ed. ), pp. 233--257. Academic Press, 
New York, 1971.
%
%

\bibitem{Sib70}
M. Sibony, 
M\'ethodes it\'eratives pour les \'equations et in \'equations aux d\'eriv\'ees partielles non lin\'eaires de type monotone, 
{\em Calcolo}, vol. 7, pp. 65?183, 1970.

\bibitem{Tseng90}
P. Tseng,
Further applications of a splitting algorithm to decomposition 
in variational inequalities and convex programming, 
{\em Math. Programming}, vol. 48, pp. 249--263, 1990.

\bibitem{Tseng91}
P. Tseng,  
Applications of a splitting algorithm to
decomposition in convex programming and variational
inequalities,
{\em SIAM J. Control Optim.},
{vol. 29}, 
pp. 119--138, 1991.

\bibitem{VilRosMos14}
S. Villa, L. Rosasco, S. Mosci, and A. Verri,
Proximal methods for the latent group lasso penalty,
{\em Comput. Anal. Optim.}, vol. 58, pp. 381--407, 2014.

\bibitem{VilSal13}
S. Villa, S. Salzo, L. Baldassarre, and A. Verri, 
{Accelerated and inexact forward-backward algorithms},
{\em SIAM J. Optim.}, vol. 23, pp. 1607--1633, 2013.

\bibitem{aicm1}
B. C.  V\~u, 
A splitting algorithm for dual monotone inclusions 
involving cocoercive operators. 
{\em Adv. Comput. Math.},
{vol. 38},
pp. 667--681, 2013.

\bibitem{jota1}
B. C.  V\~u,
A splitting algorithm for coupled system of primal-dual monotone inclusions, 
{\em J. Optim. Theory Appl.}, vol. 164 , 993--1025, 2015

\bibitem{Zhud96}
D. L. Zhu and P. Marcotte, 
Co-coercivity and its role in the convergence of iterative schemes 
for solving variational inequalities,
{\em SIAM J. Optim.,}
vol. 6, pp. 714--726, 1996. 
\end{thebibliography}
\end{document}